\documentclass{amsart}
\usepackage{amsmath}
\usepackage{color}
\usepackage{mathtools}
\usepackage[colorlinks,
            linkcolor=blue,
            anchorcolor=blue,
            citecolor=red,
            bookmarksnumbered]{hyperref}

\newcommand{\p}{\partial}

\newcommand{\ol}{\overline}
\newcommand{\ul}{\underline}

\newtheorem{theorem}{Theorem}[section]
\newtheorem{lemma}[theorem]{Lemma}

 \theoremstyle{definition}
\newtheorem{definition}[theorem]{Definition}
\newtheorem{example}[theorem]{Example}

\theoremstyle{remark}
\newtheorem{remark}[theorem]{Remark}

\numberwithin{equation}{section}

\begin{document}

\title[Pogorelov Estimates]
{The Pogorelov estimates for
degenerate curvature equations}

\author{Heming Jiao}
\address{School of Mathematics and Institute for Advanced Study in Mathematics, Harbin Institute of Technology,
         Harbin, Heilongjiang 150001, China}
\email{jiao@hit.edu.cn}
\thanks{The first author is supported by the NSFC (Grant No. 12271126),
the Natural Science Foundation of Heilongjiang Province (Grant No. YQ2022A006),
and the Fundamental Research Funds for the Central Universities (Grant No. HIT.OCEF.2022030).}

\author{Yang Jiao}
\address{School of Mathematics, Harbin Institute of Technology,
         Harbin, Heilongjiang 150001, China}
\email{18b912015@stu.hit.edu.cn}



\begin{abstract}

We establish the Pogorelov type estimates for degenerate prescribed $k$-curvature equations as well as $k$-Hessian equations.
Furthermore, we investigate the interior $C^{1,1}$ regularity of the solutions for Dirichlet problems.
These techniques also enable us to improve the existence theorem for an asymptotic Plateau type problem in hyperbolic space.


\end{abstract}

\maketitle

\section{Introduction}

In this paper, we establish the Pogorelov type estimates for the solution of the Dirichlet problem for degenerate
$k$-curvature equations
\begin{equation}
\label{1-1}
\left\{ \begin{aligned}
   \sigma_k (\kappa [M_u]) & = f (x, u)\geq 0 \;\;\mbox{ in }~ \Omega, \\
                 u &= \varphi  \;\;\mbox{ on }~ \partial \Omega,
\end{aligned} \right.
\end{equation}
where $\Omega$ is a bounded domain in $\mathbb{R}^n$, $M_u=\{(x,u(x)); x\in\Omega\}$ is the graphic hypersurface defined by a function $u$,
$\kappa[M_u] = (\kappa_1, \ldots, \kappa_n)$ represents the principal curvatures of $M_u$ and $\sigma_{k}$ denotes the $k$-th elementary symmetric functions
\[
\sigma_{k} (\kappa) = \sum_ {i_{1} < \ldots < i_{k}}
\kappa_{i_{1}} \ldots \kappa_{i_{k}},\ \ k = 1, \ldots, n.
\]
It is worth noting that when $k=1, 2$ and $n$, the left-hand side of \eqref{1-1} represents the mean, scalar, and
Gauss curvature of the hypersurface $M_u$, respectively.

The degenerate curvature equations arise in many geometric problems, including the degenerate Weyl problem, which was extensively studied by Guan-Li \cite{GL1} and Hong-Zuily \cite{HZ},
 and the prescribed Gauss curvature measure problem, which was investigated by Guan-Li \cite{GL2}.
All of these problems can be reformulated as some degenerate Gauss curvature type curvature equations.
One of our motivations for studying the Dirichlet problem \eqref{1-1} is its relevance to the Plateau-type problems,
which can be locally expressed in the same form of \eqref{1-1} (as showing in \cite{GS02,GS04,TW02}).

Non-degenerate curvature equations have been extensively studied  by numerous authors,
(see for instance \cite{CNSIV, CNSV, GS04, GRW15, I90, I91, ILT96, LT94b, RW19, RW20, STW12, SUW12, SX17, T90, T90b, U00} and the references therein.)

Unlike the non-degenerate case, where the regularity is well-established,
the best regularity for the degenerate equations we can expect is $C^{1,1}$.
Recently, Jiao-Wang \cite{JW21} established the existence of solutions in $C^{1,1} (\ol \Omega)$
for equation \eqref{1-1} with homogenous boundary conditions, provided
\begin{equation}
\label{JJW-1}
f^{1/(k-1)} \in C^{1,1} (\mathbb{R}^n \times \mathbb{R})
\end{equation}
holds in a uniformly $k$-convex domain $\Omega$.
Another class of degenerate curvature type equations which is the combination of $\sigma_k$, has been considered by Guan-Zhang \cite{GZ}.

In the current work, we focus on the interior regularity of \eqref{1-1}.
In particular, we establish the Pogorelov-type estimates for admissible solutions of \eqref{1-1} under the condition
\begin{equation}
	\label{condition}
	f^{1/(k-1)} \in C^{1,1} (\ol \Omega\times\mathbb{R}).
\end{equation}

\begin{definition}\label{k-convex} For a domain $\Omega\subset \mathbb R^n$, a function $v\in C^2(\Omega)$ is called $k$-convex if the eigenvalues $\lambda (x)=(\lambda_1(x), \cdots, \lambda_n(x))$ of the Hessian $D^2 v(x)$ lie in $\ol{\Gamma}_k$ for all $x\in \Omega$, where $\Gamma_k$ is the G\r{a}rding's cone
\begin{equation}\label{def G2}
\Gamma_k=\{\kappa \in \mathbb R^n \ | \quad \sigma_m(\kappa)>0,
\quad  m=1,\cdots,k\}.\nonumber
\end{equation}
\end{definition}
\begin{definition}
	A $C^2$ regular hypersurface $\Sigma\subset \mathbb{R}^{n+1}$ is called $k$-convex if its principal curvature vector
	$\kappa(X)\in \ol{\Gamma}_k$ for all $X\in \Sigma$. A function $u \in C^2 (\Omega)$ is called admissible if its graph
	is $k$-convex.
\end{definition}
It is well known that equation \eqref{1-1} is elliptic with respect to admissible solutions.

To prove the existence of admissible solutions to the Dirichlet problem \eqref{1-1}, it is necessary to assume some geometric
restrictions on the domain. A bounded domain $\Omega$ in $\mathbb{R}^n$ is called (uniformly) $(k-1)$-convex
if there exists a positive constant
$K$ such that for each $x \in \partial \Omega$,
\[
(\kappa^b_1 (x), \ldots, \kappa^b_{n - 1} (x), K) \in \overline{\Gamma}_{k} (\Gamma_{k}),
\]
where $\kappa^b_1 (x), \ldots, \kappa^b_{n - 1} (x)$ are the principal curvatures of $\partial \Omega$ at
$x$. We note that the $(n-1)$-convexity is equivalent to the usual convexity.

Our main result is as follows.
\begin{theorem}
\label{interior}
Suppose that $\Omega$ is a bounded domain in $\mathbb{R}^{n}$, and $f > 0$ satisfies \eqref{condition}.
Let $u \in C^4 (\Omega) \cap C^1 (\ol \Omega)$ be an admissible solution of \eqref{1-1}.
Assume that $\ol u \in C^{1,1} (\ol \Omega)$ is $(k+1)$-convex if $k<n$ or convex if $k=n$ satisfying
\[
\ol u > u \mbox{ in } \Omega \mbox{ and } \ol u = u \mbox{ on } \partial \Omega.
\]
Then the second fundamental form $h$
of the graphic hypersurface $M_u$ satisfies the estimates
\begin{equation}\label{curvature}
  |h|\leq\frac{C}{(\ol u-u)^{\alpha}},
\end{equation}
where $\alpha=\sup\{3, k-1\}$ and the constant $C$ depends only on $n$, $k$, $\|u\|_{C^1 (\ol \Omega)}$, $\|\ol u\|_{C^{1,1} (\ol \Omega)}$ and $\|f^{1/(k-1)}\|_{C^{1,1} (\overline{\Omega} \times [\inf_{\overline{\Omega}} u, \sup_{\overline{\Omega}} u])}$
, but is independent of the lower bound of $f$.
\end{theorem}

If the principal curvatures in \eqref{1-1} are replaced by eigenvalues of the Hessian, $D^2 u$, equation \eqref{1-1} becomes the
$k$-Hessian equation
\begin{equation}
\label{1-1h}
\left\{ \begin{aligned}
   \sigma_k (\lambda (D^2 u)) & = f (x, u) \;\;\mbox{ in }~ \Omega, \\
                 u &= \varphi  \;\;\mbox{ on }~ \partial \Omega,
\end{aligned} \right.
\end{equation}
where $\lambda(D^{2}u)=(\lambda_{1}, \cdots, \lambda_{n})$ denote the eigenvalues of $D^{2}u$.
When $k=n$, \eqref{1-1h} corresponds to the Monge-Amp\`{e}re equation.

Pogorelov's interior second order estimate was originally introduced by Pogorelov \cite{Pogorelov} for
non-degenerate Monge-Amp\`{e}re equation.
His results were extended to $k$-Hessian equations by Chou-Wang \cite{CW01}
and to curvature equations by Sheng-Urbas-Wang \cite{SUW12}.
Liu-Trudinger \cite{LiuTrudinger10} and Jiang-Trudinger \cite{JiangTrudinger14} established the Pogorelov estimates for Monge-Amp\`{e}re type
equations arising from optimal transportation and geometrical optics.
Li-Ren-Wang \cite{LRW16} proved the Pogorelov estimates for 2-convex solutions of 2-Hessian equations,
with the function on the right-hand side also depending on the gradient of the solution.
When $k>2$, similar estimates
were also established for $(k+1)$-convex solutions of general $k$-Hessian
equations there.
Chu-Jiao \cite{CJ21} considered the Pogorelov estimates for a class of Hessian-type equations, which include
the $(n-1)$ Monge-Amp\`{e}re equation.
Their results were further generalized by Chen-Tu-Xiang \cite{CTX21} later on.
Dinew \cite{Dinew20}
considered another generalization of the $(n-1)$ Monge-Amp\`{e}re equation.
Another interesting equations are called sum Hessian equations
which are combinations of Hessian operators $\sigma_k$.
Recently, Liu-Ren \cite{LR22} studied the Pogorelov estimates for such equations.

All the aforementioned works dealt with the non-degenerate cases.
Blocki \cite{Blocki03b} studied the interior regularity for the degenerate Monge-Amp\`{e}re equation and established the Pogorelov estimates.
Guan-Trudinger-Wang \cite{GTW99} proved the $C^{1,1}$ regularity up to the boundary
for degenerate Monge-Amp\`{e}re equations, as well as prescribed
Gauss curvature equations provided that
\begin{equation}
\label{GTW}
f^{1/(n-1)} \in C^{1,1} (\ol \Omega)
\end{equation}
holds.
It is natural to ask whether one can extend the results of \cite{GTW99} to the $k$-Hessian equations, aiming to establish the $C^{1,1}$ estimates both in the interior and on the boundary, under the condition
\begin{equation}
\label{ITW}
f^{1/(k-1)} \in C^{1,1} (\ol \Omega).
\end{equation}
This question was proposed by Ivochkina-Trudinger-Wang \cite{ITW2004} and remains open until now.
An example presented by Wang \cite{Wang95} indicates that the condition \eqref{ITW} should be
optimal for the existence of solutions in $C^{1,1} (\ol \Omega)$.
For degenerate Hessian equations, Krylov established the
the $C^{1, 1}$ estimates in \cite{K94a, K94b, K95a, K95b} under the assumption
$f^{1/k} \in C^{1,1} (\ol \Omega)$ and his proof was further simplified
by Ivochkina-Trudinger-Wang \cite{ITW2004}.
Write $\widetilde{f} := f^{1/(k-1)}$. Dong \cite{Dong06} considered \eqref{1-1h} with $\varphi\equiv 0$.
He derived the $C^{1,1}$ estimates under a slightly stronger condition
\begin{equation}
\label{strong}
|D \widetilde{f}| \leq C \widetilde{f}^{1/2} \mbox{ on } \ol \Omega
\end{equation}
for some positive constant $C$.

The Pogorelov estimates Theorem \ref{interior} can be applied to the Dirichlet problem with affine boundary values.
\begin{theorem}
\label{cor2}
Suppose $\Omega$ is a bounded domain in $\mathbb{R}^n$ and $f$ satisfies $f \geq 0$, $f_u \geq 0$,
and \eqref{condition}.
Assume that $\varphi$ is affine, and there exists a sub-solution $\underline{u} \in C^{1,1} (\ol \Omega)$ satisfying $\kappa [M_{\ul u}] \in \Gamma_k$ and
\begin{equation}
\label{subsol}
\left\{ \begin{aligned}
   \sigma_k(\kappa[M_{\underline{u}}]) & \geq f(x,\underline{u}) & \;\;\mbox{ in }~ &\Omega,\\ 
                 \ul u &= \varphi  & \;\;\mbox{ on }~ & \partial \Omega.
\end{aligned} \right.
\end{equation}
Then, there exists a unique admissible solution
$u\in C^{1,1}(\Omega) \cap C^{0}(\ol \Omega)$ of the Dirichlet problem \eqref{1-1}.
\end{theorem}

Our approach for proving Theorem \ref{interior} is inspired partially by \cite{CW01} and \cite{SUW12}.
Compared to the non-degenerate cases explored in \cite{CW01} and \cite{SUW12},
the main difficulty in the current work is from the degeneracy of the equations and the condition \eqref{condition},
instead of that $f^{1/k} \in C^{1,1} (\ol \Omega \times \mathbb{R})$.
To overcome these difficulties, we make use of a lemma proved by B{\l}ocki 
\cite{Blocki03} (Lemma \ref{Blocki}) and the generalized Newton-MacLaurin inequality \eqref{nm}
for elementary symmetric functions.

To clarify our approach, we first establish the Pogorelov estimates for degenerate $k$-Hessian equations.
In contrast to prescribed curvature equations,
it is possible to construct a sub-solution based on the $(k-1)$-convexity of the domain, as demonstrated in \cite{CNS}.
Our application of the Pogorelov estimates to degenerate $k$-Hessian equations \eqref{1-1h} leads to the following existence theorem
when  $\varphi$ is affine.
\begin{theorem}
\label{cor1}
Suppose $\Omega$ is uniformly $(k-1)$-convex with $\partial \Omega \in C^{1, 1}$.
Assume that $\varphi$ is affine,
$f \geq 0$, $f_{u}\geq 0$, and $f^{1/(k-1)} \in C^{1,1} (\Omega \times \mathbb{R}) \cap C^0 (\ol \Omega \times \mathbb{R})$.
Then, there exists a unique $k$-admissible solution $u\in C^{1,1}(\Omega)\cap C^{0}(\ol \Omega)$ to the Dirichlet problem \eqref{1-1h}.
\end{theorem}

For general $\varphi$, analogous to Theorem 1.2 in \cite{Blocki03b} where the Monge-Amp\`{e}re equation was studied, we can establish:
\begin{theorem}
\label{hessian}
Suppose $\Omega$ is uniformly $(k-1)$-convex with $C^{3,1}$ boundary. Assume that $f\geq 0$, $f^{1/(k-1)} \in C^{1,1} (\ol \Omega)$,
$\varphi \in C^{3,1} (\ol \Omega)$, and
\begin{equation}
\label{right}
\mbox{ every connected component of the set } \{f=0\} \cap \Omega \mbox{ is compact.}
\end{equation}
Then there exists a unique solution $u \in C^{1,1} (\Omega) \cap C^{0,1} (\ol \Omega)$ of \eqref{1-1h}.
\end{theorem}

An interesting problem in hyperbolic geometry so called the asymptotic Plateau type problem
aims to find a complete hypersurface with prescribed curvature and asymptotic boundary at infinity.
When the prescribed curvature is constant, this problem was studied extensively by Guan-Spruck \cite{GS00, GS10}, Guan-Spruck-Szapiel \cite{GSS09}
and Guan-Spruck-Xiao \cite{GSX14}. Sui \cite{Sui19}, Sui-Sun \cite{SuiSun22a, SuiSun22} extended their results to the non-constant case
by introducing a new approximate problem using the level set of a known subsolution. It is worth mentioning that the Pogorelov estimates
play a key role in their proof so that our techniques make it possible to improve Sui-Sun's results. The details are
stated in Section 6.

The rest of the paper is organized as follows.
In Section 2, we present some
preliminaries which may be used in the following sections.
Pogorelov type estimates
for k-Hessian equations and prescribed k-curvature equations are derived in Section
3 and Section 4 respectively.
The Dirichlet problems \eqref{1-1} and \eqref{1-1h} are considered in Section 5.
In Section 6, we are concerned with the asymptotic Plateau type problem in hyperbolic spaces.



\bigskip

\textbf{Acknowledgement.} The authors wish to thank Zhenan Sui for pointing out Theorem \ref{Theorem6'},
which was not included in an early version of the paper and many valuable discussions.

\section{Preliminaries}

Throughout this paper, $v_i = \frac{\p v}{\p x_i}$, $v_{ij} = \frac{\p^2 v}{\p x_i \p x_j}$,
$Dv=(v_{1}, \cdots, v_{n})$ and $D^2 v = (v_{ij})$ denote the ordinary first and second order derivatives, gradient and Hessian of a function $v \in C^2 (\Omega)$ respectively.
Moreover, we denote $\partial_i f (x, v(x))=\frac{\p f(x,v (x))}{\p x_i}=f_{i} (x, v(x))+f_{v} (x, v(x))v_{i} (x)$ for a
function $f$ defined on $\Omega \times \mathbb{R}$.

As in \cite{JW21}, a graphic hypersurface $M_u$ in $\mathbb{R}^{n+1}$ is a codimension one submanifold which can be written as a graph
\[ M_u=\{X=(x, u(x)): x\in\mathbb{R}^n\}.\]
Let
$\epsilon_{n+1} = (0, \cdots, 0, 1) \in \mathbb{R}^{n+1}$, then the height function of $M_u$ is $u(x)=\langle X, \epsilon_{n+1}\rangle$.
The induced metric and second fundamental form of $M_u$ are given
by
$$g_{ij}=\delta_{ij}+ u_i u_j, \ \  1\leq i,j\leq n,$$
and
\[h_{ij}=\frac{u_{ij}}{\sqrt{1+|Du|^2}}\]
respectively, while the upward unit normal vector field to $M_u$ is
\[\nu=\frac{(-Du, 1)}{\sqrt{1+|Du|^2}}.\]
By straightforward calculations, we have the principle curvatures of $M_u$ are eigenvalues of the symmetric matrix
$A=(a_{ij})$:
\begin{equation}
\label{matrix}
a_{ij}=\frac{1}{w}\gamma^{ik}u_{kl}\gamma^{lj},
\end{equation}
where $\gamma^{ik}=\delta_{ik}-\frac{u_iu_k}{w(1+w)}$ and $w=\sqrt{1+|Du|^2}.$ Note that $(\gamma^{ij})$ is invertible with inverse
$\gamma_{ij}=\delta_{ij}+\frac{u_iu_j}{1+w},$ which is the square root of $(g_{ij})$.


Let $\{e_1,e_2,\cdots,e_n\}$ be a local orthonormal frame on $TM_u$. We will use $\nabla$ to denote
the induced Levi-Civita connection on $M_u$. For a function $v$ on $M_u$, we denote $\nabla_i v=\nabla_{e_i}v,$
$\nabla_{ij} v = \nabla^2 v (e_i, e_j),$ etc in this paper.
Thus, we have
\[|\nabla u|=\sqrt{g^{ij}u_{i}u_{j}}=\frac{|Du|}{\sqrt{1+|Du|^2}}.\]

Using normal coordinates, we also have the following well known fundamental equations for a hypersurface $M$ in $\mathbb{R}^{n+1}:$
\begin{equation}\label{Gauss}
\begin{aligned}
\nabla_{ij} X = \,& h_{ij}\nu  \quad {\rm (Gauss\ formula)}\\
\nabla_i \nu= \,& -h_{ij} e_j \quad {\rm (Weigarten\ formula)}\\
\nabla_k h_{ij} = \,& \nabla_j h_{ik} \quad {\rm (Codazzi\ equation)}\\
R_{ijst} = \,& h_{is}h_{jt}-h_{it}h_{js}\quad {\rm (Gauss\ equation)},
\end{aligned}
\end{equation}
where $X = (X_1, \ldots, X_{n+1})$ is the position vector field on $M$ and $\nu = (\nu_1, \ldots, \nu_{n + 1})$ is the unit normal vector field to $M$ in $\mathbb{R}^{n+1}$,
$h_{ij} = \langle D_{e_i} e_j, \nu\rangle$ is the second fundamental form,
$R_{ijst}$ is the $(4,0)$-Riemannian curvature tensor of $M$, and the derivative here is covariant derivative with respect to the metric on $M$.
Therefore, the Ricci identity becomes,
\begin{equation}\label{Commutation formula}
\nabla_{ij} h_{st} = \nabla_{st} h_{ij}+(h_{mt}h_{sj}-h_{mj}h_{st})h_{mi}+(h_{mt}h_{ij}-h_{mj}h_{it})h_{ms}.
\end{equation}

We need the following lemma which was proved by B{\l}ocki \cite{Blocki03}.
\begin{lemma}
\label{Blocki}
Let $D$ be a domain in $\mathbb{R}^{N}$ and $\psi \in C^{2} (\ol D)$ be nonnegative.
We have
\begin{equation}
\label{Blocki-1}
|D \psi (x)| \leq \max \left\{\frac{|D \psi (x)|}{\mathrm{dist} (x, \partial D)},
   (1 + \sup_D \lambda_{\mathrm{max}} (D^2 \psi))\right\} \sqrt{\psi (x)}
\end{equation}
for all $x \in D$.
\end{lemma}


In this paper we denote $\sigma_{m; i_1, \ldots, i_s} (\kappa) = \sigma_m (\kappa) \mid_{\kappa_{i_1} = \cdots = \kappa_{i_s} = 0}$
for $\{i_1, \ldots, i_s\} \subset \{1, \ldots, n\}$, $1 \leq m \leq n$ and $n-s \leq m$. It is easy to see
\[
\frac{\partial \sigma_k}{\partial \kappa_i} (\kappa) = \sigma_{k-1; i} (\kappa).
\]
For a symmetric matrix $A=\{a_{ij}\}$ with $\lambda (A) \in \Gamma_k$, we denote
\[
\sigma_k^{ij}(A) := \frac{\partial \sigma_k (\lambda (A))}{\partial a_{ij}}
\]
and
\[
\sigma_k^{ij, pq}(A) := \frac{\partial^2 \sigma_k (\lambda (A))}{\partial a_{ij} \partial a_{pq}}.
\]
It is well known that the matrix $(\sigma_k^{ij})$ is positive definite.

We need the following generalized Newton-MacLaurin inequality. Its proof can be found in \cite{Spruck05}.
\begin{lemma}
  For $\lambda\in \Gamma_{k}$ and $k>l\geq 0, r>s\geq 0, k\geq r, l\geq s$, we have
\begin{equation}\label{nm}
  \Bigg[\frac{\sigma_{k}(\lambda)/C^{k}_{n}}{\sigma_{l}(\lambda)/C^{l}_{n}}\Bigg]^{\frac{1}{k-l}}
  \leq \Bigg[\frac{\sigma_{r}(\lambda)/C^{r}_{n}}{\sigma_{s}(\lambda)/C^{s}_{n}}\Bigg]^{\frac{1}{r-s}}.
\end{equation}
Here we define $\sigma_0 \equiv 1$.
\end{lemma}
The following lemma is Lemma 3.1 in \cite{CW01}.
\begin{lemma}
\label{lem-3}
  Suppose $\lambda\in \Gamma_{k}$ and $\lambda_{1}\geq \cdots \geq \lambda_{n}$. For any $\overline{\delta} \in (0, 1)$, if there exist some $ \epsilon >0$
  such that

  (i) $\sigma_{k}(\lambda)\leq \epsilon \lambda_{1}^{k}$ \\
  or

  (ii) $|\lambda_{i}|\leq \epsilon\lambda_{1}$ for $i=k+1, \ldots, n$,\\
  then we have $\lambda_{1}\sigma_{k-1;1}\geq (1-\overline{\delta})\sigma_{k}$.
\end{lemma}
The following lemma can be regarded as another form of Lemma 3.1 in \cite{I91}.
\begin{lemma}
\label{lem-4}
Suppose the function $v \in C^2 (\Omega)$ is $(k+1)$-convex. If we regard $v$ as a function defined on the hypersurface
$M_u$, i.e., $v(X) = v(x)$ for $X = (x, u(x))\in M_{u}$, then we have
\begin{equation}
\label{add-1}
\lambda (\nabla^2 v - \nu (v) h) (X) \in \ol{\Gamma}_k, \mbox{ for any } X \in M_{u},
\end{equation}
where $\nabla$ is the Levi-Civita connection on $M_u$,
$h$ is the second fundamental form of $M_u$ and $\nu (v)$ is defined by $\nu (v) := \sum_{i=1}^n \nu_i v_i$.
\end{lemma}
\begin{proof}
Let $\{\epsilon_1, \ldots, \epsilon_{n+1}\}$ be the natural basis of $\mathbb{R}^{n+1}$
and $\{e_1,\ldots,e_n\}$ be orthonormal frame on $TM_u$ given by
\[
e_i := \sum_{j=1}^n \gamma^{ij} (\epsilon_j + u_j \epsilon_{n+1}), \ \ i = 1, \ldots, n,
\]
where $\gamma^{ij}=\delta_{ij}-\frac{u_iu_j}{w(1+w)}$. Let $\tau$ be the matrix $\{\gamma^{ij}\}$.
By Lemma 3.1 of \cite{I91} (seeing Proposition 2.1 in \cite{JW21} also), we have $\lambda (\tau D^2 v \tau) \in \ol{\Gamma}_k$
and furthermore
\begin{equation}
\label{add-10}
\sigma_j (\lambda (\tau D^2 v \tau)) \geq \frac{1}{w^2} \sigma_j (\lambda (D^2 v)), \ \ j = 1, \ldots, k.
\end{equation}
We regard $v$ as a function defined on a neighborhood of $M_u$ in $\mathbb{R}^{n+1}$ by $v (x, x_{n+1}) \equiv v (x)$.
Therefore, we have
\[
\begin{aligned}
  \nabla_{ij}v
  = \,& \nabla^2 v (e_i, e_j)
  = D^2 v (e_i, e_j) + \nu (v) h_{ij}\\
      = \,& \sum_{s,t=1}^{n}\gamma^{is} v_{st} \gamma^{jt} + \nu (v) h_{ij},
\end{aligned}
\]
where $v_{st} = \frac{\partial^2 v}{\partial x_s \partial x_t}$ denote the Euclidean derivatives.
Thus, by \eqref{add-10}, we obtain \eqref{add-1}.
\end{proof}

\section{Pogorelov estimates for Hessian equation}

In this section we establish the Pogorelov estimate for degenerate $k$-Hessian equations.

\begin{theorem}
\label{interior_h}
Suppose $\Omega$ is bounded and $f > 0$ satisfies \eqref{condition}.
Assume that there is a $k$-convex function
$\ol u \in C^{1,1} (\ol \Omega)$ such that
$$
\ol u \geq u \;\mbox{in}\; \Omega\;\;\mbox{and}\;\;\ol u=u \;\mbox{on}\; \partial\Omega.
$$
Then the $k$-convex solution $u \in C^4 (\ol \Omega)$
of the Dirichlet problem \eqref{1-1h} satisfies the estimates
\begin{equation}
\label{interior-1}
(\ol u - u)^{\alpha} |D^2 u| \leq C,
\end{equation}
where $\alpha=k-1$ for $k>2$ and $\alpha=2$ for $k=2$, $C$ depends only on $n$, $k$, $\|u\|_{C^1 (\ol \Omega)}$,
$\|\ol u\|_{C^{1,1} (\ol \Omega)}$ and $\|f^{1/(k-1)}\|_{C^{1,1} (\ol \Omega \times [\inf u, \sup u])}$, but is independent of the lower bound of $f$.
\end{theorem}
\begin{proof}
Let
\[
H (x, \xi) = \rho^{\alpha} (x) \exp\{\frac{\delta}{2} |D u (x)|^2+ \frac{b}{2} |x|^2\} u_{\xi \xi} (x),
\]
where $\rho = \ol u - u$, $|\xi| = 1$, $\delta$ and $b$ are positive constants to be determined.
Suppose the maximum of $H$ is attained at $x_0 \in \Omega$ and $\xi = (1, 0, \ldots, 0)$. We may assume
$D^2 u$ is diagonal and $u_{11} \geq \cdots \geq u_{nn}$ at $x_0$ by a rotation of axes if necessary.
Therefore, the function
\[
\alpha \log \rho + \frac{\delta}{2} |D u (x)|^2 + \frac{b}{2} |x|^2 + \log u_{11},
\]
also achieves its maximum at $x_0$ where
we have
\begin{equation}
\label{interior-2}
0  = \frac{\alpha\rho_i}{\rho} + \delta u_i u_{ii} + b x_i + \frac{u_{11i}}{u_{11}} \mbox{ for each } i = 1, \ldots, n
\end{equation}
and
\begin{equation}
\label{interior-3}
0 \geq  \sigma_k^{ii} \left\{\frac{\alpha\rho_{ii}}{\rho} - \alpha\frac{\rho_{i}^2}{\rho^2} + \delta u_{ii}^2 + \delta u_l u_{lii} + b
      + \frac{u_{11ii}}{u_{11}} - \frac{u_{11i}^2}{u_{11}^2}\right\}.
\end{equation}
Differentiating the equation \eqref{1-1h} twice we get
\begin{equation}
\label{interior-4}
|\sigma_k^{ii} u_{lii}| = |\partial_{l}f|=|f_l+f_{u}u_{l}| \leq C f^{1 - 1/(k-1)}, \mbox{ for each } l = 1, \ldots, n
\end{equation}
and
\begin{equation}
\label{interior-5}
\sigma_k^{ii} u_{11ii} = \partial_{11}f - \sigma_k^{ij, pq} u_{ij1} u_{pq1}
                  \geq - C u_{11} f^{1 - 1/(k-1)} - \frac{(\partial_{1}f)^2}{f}- \sigma_k^{ij, pq} u_{ij1} u_{pq1}
\end{equation}
provided $u_{11}$ is sufficiently large by the fact $f^{1/(k-1)} \in C^{1,1} (\ol \Omega\times \mathbb{R})$.
Since $D^2 u (x_0)$ is diagonal, we have, at $x_0$,
\begin{equation}
\label{add-5}
\sigma_k^{ij, pq} = \left\{ \begin{aligned}
   \sigma_{k - 2; ip} & \;\;\mbox{ if }~ i=j, p=q, i \neq q, \\
   - \sigma_{k - 2; ip} & \;\;\mbox{ if }~ i=q, j=p, i \neq j, \\
   0 & \;\;\mbox{ otherwise.}
\end{aligned} \right.
\end{equation}
Because $D^2 u(x_0)$ is diagonal again, we find $\sigma_k^{ii} = \sigma_{k-1;i}$ at $x_0$.
By the concavity of $\sigma_k^{1/k}$ in $\Gamma_k$, we have
\[
\sum_{i\neq j} \sigma_{k-2; ij} u_{ii1} u_{jj1}
  \leq \left(1- \frac{1}{k}\right) \frac{1}{f} (\sigma_{k-1;i} u_{ii1})^2 = \left(1-\frac{1}{k}\right) \frac{(\partial_{1}f)^2}{f}.
\]
It follows that
\begin{equation}
\label{interior-6}
\begin{aligned}
- \sigma_k^{ij, pq} u_{ij1} u_{pq1}
  = \,& \sum_{i \neq j} \sigma_{k-2; ij} u_{ij1}^2 - \sum_{i\neq j} \sigma_{k-2; ij} u_{ii1} u_{jj1}\\
  \geq \,& \sum_{i \neq j} \sigma_{k-2; ij} u_{ij1}^2 - C \frac{(\partial_{1}f)^2}{f}\\
  \geq \,& \sum_{i \geq 2} 2\sigma_{k-2; i1} u_{11i}^2 - C \frac{(\partial_{1}f)^2}{f}.
\end{aligned}
\end{equation}
Applying Lemma \ref{Blocki} to $\psi = f^{1/(k-1)}$ in the domain $D := \Omega \times (-\sup |u|-1, \sup |u|+1)$, we get
\begin{equation}
\label{interior-7}
\frac{(\partial_{1} f)^2}{f} \leq \frac{C}{d^2} f^{1- 1/(k-1)},
\end{equation}
where $d = d (x) := \mathrm{dist} (x, \partial \Omega)$.
Combining \eqref{interior-3}-\eqref{interior-7}, we obtain
\begin{equation}
\label{interior-8}
\begin{aligned}
0 \geq \,& \delta \sigma_k^{ii} u_{ii}^2 + b \sum \sigma_k^{ii} + \frac{2}{u_{11}}\sum_{i \geq 2} \sigma_{k-2; i1} u_{11i}^2 - C \delta f^{1-1/(k-1)}\\
  & + \alpha \sigma_k^{ii} \left(\frac{\rho_{ii}}{\rho} - \frac{\rho_{i}^2}{\rho^2}\right) - \sigma_k^{ii} \frac{u_{11i}^2}{u_{11}^2}
  - Cf^{1- 1/(k-1)} - \frac{C}{u_{11} d^2} f^{1- 1/(k-1)}.
\end{aligned}
\end{equation}
Note that there exist a constant $B$ depending only on $|\ol u-u|_{C^1 (\ol \Omega)}$ such that
\begin{equation}
\label{super-1}
0 \leq \ol u - u \leq B d \mbox{ in } \Omega.
\end{equation}
Therefore we may assume $u_{11} d^2 \geq 1$ for otherwise we are done. It follows from \eqref{interior-8} that
\begin{equation}
\label{interior-9}
\begin{aligned}
0 \geq \,& \delta \sigma_k^{ii} u_{ii}^2 + b \sum \sigma_k^{ii} + \frac{2}{u_{11}}\sum_{i \geq 2} \sigma_{k-2; i1} u_{11i}^2\\
  & + \alpha \sigma_k^{ii} \left(\frac{\rho_{ii}}{\rho} - \frac{\rho_{i}^2}{\rho^2}\right) - \sigma_k^{ii} \frac{u_{11i}^2}{u_{11}^2}
  - C f^{1- 1/(k-1)}.
\end{aligned}
\end{equation}
By the fact that $\sigma_k^{1/k}$ is concave for $k$-convex functions and homogeneous of degree one, we have
$\sigma_k^{ii} \ol u_{ii} \geq 0$
and
\begin{equation}
\label{add-2}
\sigma_k^{ii} \rho_{ii} \geq - \sigma_k^{ii} u_{ii} = - k f.
\end{equation}
Now we consider two cases as in \cite{CW01}.
\\ \hspace*{\fill} \\
\noindent
\textbf{Case 1.} $u_{kk}\geq \varepsilon u_{11}$, where $\varepsilon$ is a positive constant to be chosen.
It follows that
\begin{equation}\label{case-1_h}
  \sum \sigma_k^{ii}u_{ii}^{2}>\sigma^{kk}_{k}u_{kk}^{2}\geq \theta u_{11}^{2}\sum \sigma_k^{ii},
\end{equation}
where $\theta=\theta(n,k,\varepsilon)$.

By \eqref{interior-2}, we have
\begin{equation}
\label{add-3}
\frac{u_{11i}^2}{u_{11}^2}\leq C(\frac{\rho_{i}^{2}}{\rho^{2}}+\delta^{2}u_{ii}^{2}+b^{2}) \mbox{ for } i =1, \ldots, n.
\end{equation}

Note that $\sigma_{1}\geq u_{11}\geq \frac{1}{n}\sigma_{1}$ since $u$ is $k$-convex and $u_{11}$ is the largest eigenvalue.
Using the generalized Newton-MacLaurin inequality \eqref{nm} with $l= k-1$, $r=k$ and $s=1$, we have
\begin{equation}
\label{GC2-9}
\sum \sigma_k^{ii} = (n - k + 1) \sigma_{k - 1} \geq c_0 \sigma_k^{1 - 1/(k-1)} \sigma_1^{1/(k-1)},
\end{equation}
where $c_{0}=(n-k+1)({C^{k-1}_{n}}/{C^{k}_{n}})({C^{k}_{n}}/{C^{1}_{n}})^{1/(k-1)}$.
By \eqref{interior-9}-\eqref{GC2-9}, we have
\begin{equation}
\label{interior-10}
\begin{aligned}
  0 \geq \,& (\delta-C\delta^{2}) \sigma_k^{ii} u_{ii}^2 + (b-Cb^{2}) \sum \sigma_k^{ii}
             + \frac{2}{u_{11}}\sum_{i \geq 2} \sigma_{k-2; i1} u_{11i}^2\\
         \,& - \alpha \frac{kf}{\rho} - C\frac{|D\rho|^{2}}{\rho^{2}}\sum \sigma_k^{ii} - C f^{1- 1/(k-1)} \\
    \geq \,& (\delta-C\delta^{2})\theta u_{11}^{2}\sum \sigma_{k}^{ii}
             + (b-Cb^{2}-\frac{C}{\rho^{2}}) \sum \sigma_k^{ii}\\
         \,& - \alpha \frac{kf}{\rho} - C f^{1- 1/(k-1)}\\
    \geq \,& \frac{1}{2}c_{0}\delta_{\theta}u_{11}^{2+1/(k-1)}f^{1-1/(k-1)}
              -\frac{C}{\rho}f^{1-1/(k-1)}-Cf^{1-1/(k-1)} \\
         \,& + \frac{1}{2}\delta_{\theta}u_{11}^{2} \sum \sigma_k^{ii}+b_{0} \sum \sigma_k^{ii}-\frac{C}{\rho^{2}} \sum \sigma_k^{ii},
\end{aligned}
\end{equation}
where $\delta_{\theta}=(\delta-C\delta^{2})\theta$, $b_{0}=b-Cb^{2}$ are positive by choosing $\delta$ and $b$ small enough respectively.
Then we get $\rho|D^{2}u|\leq C$ from \eqref{interior-10}.
\\ \hspace*{\fill} \\
\noindent
\textbf{Case 2.} $u_{kk}< \varepsilon u_{11}$.
We then have $|u_{jj}|\leq C\varepsilon u_{11}$ for $j=k, \ldots, n$. 
By \eqref{interior-2} and Cauchy-Schwarz inequality, we obtain, for any $\epsilon_0 > 0$
\begin{equation}\label{constant}
 \frac{\rho_{i}^{2}}{\rho^{2}}\leq
   \frac{1}{\alpha^{2}}\left(C \epsilon_0^{-1}(\delta^{2}u_{ii}^{2}+b^{2})
     +(1+ \epsilon_0)\frac{u_{11i}^{2}}{u_{11}^{2}}\right).
\end{equation}
By \eqref{interior-9}, \eqref{add-2}, \eqref{constant} and \eqref{interior-2}, we have
\begin{equation}
\label{interior-11}
\begin{aligned}
 0 \geq \,& (\delta-C\epsilon_0^{-1}\delta^{2}) \sigma_k^{ii} u_{ii}^2 + (b-C\epsilon_0^{-1}b^{2}) \sum \sigma_k^{ii}
             - \alpha \frac{kf}{\rho}- C f^{1- 1/(k-1)}\\
         \,& +\frac{2}{u_{11}}\sum_{i \geq 2} \sigma_{k-2; i1} u_{11i}^2
             -(1+\frac{1+\epsilon_0}{\alpha})\sum_{i=2}^{n}\sigma_{k}^{ii}\frac{u_{11i}^2}{u_{11}^{2}}  \\
         \,& -\alpha \sigma_{k}^{11}\frac{\rho_{1}^{2}}{\rho^{2}}-\sigma_{k}^{11}\Big(\frac{\alpha\rho_{1}}{\rho}+\delta u_{1}u_{11}+bx_{1}\Big)^{2}.
\end{aligned}
\end{equation}
Now we fix the constants $\epsilon_0 \in (0,1)$
and $\varepsilon$ to be the constant in Lemma \ref{lem-3} for $\overline{\delta} = \frac{\alpha - 1 - \epsilon_0}{2\alpha}$.
Applying Lemma \ref{lem-3} to $\sigma_{k-1; i}$, we get
$$
  \frac{2}{u_{11}}\sum_{i=2}^{n} \sigma_{k-2; i1} u_{11i}^2
              -\frac{\alpha+1+\epsilon_0}{\alpha}\sum_{i=2}^{n}\sigma_{k}^{ii}\frac{u_{11i}^2}{u_{11}^{2}} \geq 0.
$$
We fix $b$ and $\delta$ sufficiently small such that $b_0 := b - C \epsilon_0^{-1}b^2 > 0$ and $\delta_{0} :=\delta-C\epsilon_0^{-1}\delta^{2}>0$ in \eqref{interior-11}. Thus, by \eqref{GC2-9} and \eqref{interior-11}, we have
\begin{equation}
\label{interior-12}
 0 \geq (b_{0} u_{11}^{1/(k-1)}-  \frac{C}{\rho}
             - C )f^{1- 1/(k-1)}
              +\sigma_{k}^{11}\Big(\delta_{0}u_{11}^{2}-\frac{C}{\rho^{2}}-C\Big)
\end{equation}
and \eqref{interior-1} is proved.

\end{proof}


\begin{remark}
	\label{r2}
If $\Omega$ is
uniformly $(k-1)$-convex with $\partial \Omega \in C^{3, 1}$ and $\varphi \in C^{3,1} (\partial \Omega)$,
there exists a unique $k$-convex solution $\ol u \in C^{1,1} (\ol \Omega)$ of the degenerate $k$-Hessian equation
\begin{equation}\label{w}
	\left\{ \begin{aligned}
		\sigma_k \big(\lambda(D^{2} \ol u)\big) & = 0  \;\;\mbox{ in }~ \Omega, \\
		\ol u &= \varphi  \;\;\mbox{ on }~ \partial \Omega.
	\end{aligned} \right.
\end{equation}
The reader is referred to \cite{ITW2004} or \cite{K94a, K94b, K95a, K95b} for the existence theorem.
By the maximum principle, the solution of \eqref{w} satisfies the conditions of Theorem \ref{interior_h},
which implies that when $\Omega$ is uniformly
$(k-1)$-convex, $\ol u$ always exists,  the Pogorelov estimate
in Theorem \ref{interior_h} always holds in this case.
\end{remark}

\section{Pogorelov estimates for curvature equation}

In this section we establish the Pogorelov estimates for prescribed $k$-curvature equations, i.e.,
Theorem \ref{interior}.

\begin{proof}[Proof of Theorem \ref{interior}]

 Let $ v := 1/w = \langle \nu,\epsilon_{n+1} \rangle$. There exists a positive constant $a$ depending only on
$\|Du\|_{C^{0}(\ol \Omega)}$ such that $v\geq 2a$. We consider the test function
$$
W=\rho^{\alpha}(v-a)^{-1}\exp\{\frac{b}{2}|X|^{2}\}h_{\xi\xi},
$$
where $\rho=\ol u-u$, $X=(X_{1}, \ldots, X_{n+1})\in M_u$, $\xi \in T_X M_u$ is a unit vector and $b$ is a positive constant to be determined.
Suppose that the maximum value of $W$ is achieved at a point $X_{0} = (x_{0}, u(x_{0}))\in M$, $x_{0}\in\Omega$
and $\xi_0 \in T_{X_0} M_u$.
Let $\{e_{1}, e_{2}, \ldots, e_{n}\}$
be the normal coordinates with respect to $X_{0}$. We may assume $\xi_0 = e_1$ and
$\{h_{ij}\}$ is diagonal at $X_{0}$ with $h_{11} \geq \cdots \geq h_{nn}$ by a rotation if necessary.
Let $\kappa_i = h_{ii} (X_0)$ for $i = 1, \ldots, n$.
Therefore, at $X_{0}$ , taking the covariant derivatives twice with respect
to
\[
\alpha \log \rho - \log (v-a) + \frac{b}{2}|X|^2 + \log h_{11},
\]
we have
\begin{equation}\label{p1}
  \alpha\frac{\nabla_{i}\rho}{\rho}- \frac{\nabla_{i}v}{v-a}+b\langle X, e_{i} \rangle+\frac{\nabla_{i}h_{11}}{h_{11}}=0
\end{equation}
for $i=1, \ldots, n$, and
\begin{equation}\label{p2}
\begin{aligned}
  0\geq \,& \sigma_{k}^{ii}\Big\{\alpha\frac{\nabla_{ii}\rho}{\rho}-\alpha\frac{(\nabla_{i}\rho)^{2}}{\rho^{2}}
            - \frac{\nabla_{ii}v}{v-a}+ \frac{(\nabla_{i}v)^{2}}{(v-a)^{2}}\\
        \,& +b(1+\kappa_{i}\langle X, \nu \rangle)+\frac{\nabla_{ii}h_{11}}{h_{11}}-\frac{(\nabla_{i}h_{11})^{2}}{h_{11}^{2}}\Big\}.
\end{aligned}
\end{equation}
Since $\ol u$ is $(k+1)$-convex, by Lemma \ref{lem-4}, we have
\[
\sigma_k^{ii} \nabla_{ii} \ol u = \sigma_k^{ii} (\nabla_{ii} \ol u - \nu (\ol u) h_{ii}) + \nu (\ol u) \sigma_k^{ii} h_{ii}
  \geq k \nu (\ol u) f.
\]
Therefore, by the Gauss formula, we have
\begin{equation}\label{m1}
\begin{aligned}
  \sigma_{k}^{ii} \nabla_{ii} \rho
        = \,&\sigma_{k}^{ii}\nabla_{ii}(\ol u-X_{n+1})\\
        \geq \,& k \nu (\ol u) f - v \sigma_{k}^{ii}h_{ii} \geq -Cf.
\end{aligned}
\end{equation}
Differentiating the equation \eqref{1-1} twice we get
\begin{equation}
\label{d1}
|\sigma_k^{ii}  \nabla_{l} h_{ii}| = |\nabla_{l}f| \leq C f^{1 - 1/(k-1)}, \mbox{ for each } l = 1, \ldots, n
\end{equation}
and
\begin{equation}
\label{d2}
\sigma_k^{ii} \nabla_{11} h_{ii} = \nabla_{11}f - \sigma_k^{ij, pq} \nabla_{1}h_{ij}\nabla_{1} h_{pq}.
\end{equation}
Since $f^{1/(k-1)} \in C^{1,1} (\ol \Omega \times \mathbb{R})$, we have
\begin{equation}
\label{add-4}
\begin{aligned}
\nabla_{11}f = \,& \sum_{s,t=1}^{n+1} \frac{\partial^2 f}{\partial X_s \partial X_t} \nabla_1 X_s \nabla_1 X_t
     + h_{11} \nu (f)\\
      \geq \,& \left(1 - \frac{1}{k-1}\right) \frac{|\nabla_1 f|^2}{f} - C h_{11} f^{1 - 1/(k-1)}
\end{aligned}
\end{equation}
provided $h_{11}$ is sufficiently large.
By the Weingarten equation, we have
\begin{equation}
\label{m3}
\nabla_i v = - h_{im} \langle e_m, \epsilon_{n+1}\rangle = - h_{im} \nabla_m u
   \text{ and }
\sigma_k^{ii} (\nabla_i v)^2 \leq  \sigma_{k-1; i} \kappa_i^2.
\end{equation}
Next, by Gauss formula, Codazzi euqation and \eqref{d1}, we have
\begin{equation}
\label{m4}
\begin{aligned}
\sigma_k^{ii} \nabla_{ii} v = \,& - \sigma_k^{ii} \nabla_m h_{ii} \nabla_m u - \sigma_k^{ii} \kappa_{i}^2 v\\
  = \,& - v \sigma_{k-1; i} \kappa_i^2 - \langle\nabla f, \nabla u\rangle
  \leq - v \sigma_{k-1; i} \kappa_i^2 + C f^{1-1/(k-1)},
\end{aligned}
\end{equation}
where the last inequality comes from $f^{1/(k-1)} \in C^1 (\ol \Omega \times \mathbb{R})$.
By \eqref{Commutation formula}, \eqref{d2} and \eqref{add-4}, we see
\begin{equation}
\label{m71}
\begin{aligned}
\sigma_k^{ii} \nabla_{ii} h_{11} = \,&  \sigma_k^{ii} \nabla_{11} h_{ii} - h_{11} \sigma_{k-1; i} \kappa_i^2 + k f  h_{11}^{2} \\
  \geq \,& \nabla_{11}f -  \sigma_k^{ij, pq} \nabla_1 h_{ij} \nabla_1 h_{pq} - h_{11} \sigma_{k-1; i} \kappa_i^2 + k f  h_{11}^{2}\\
  \geq \,& - C h_{11} f^{1 - 1/(k-1)} -  \sigma_k^{ij, pq} \nabla_1 h_{ij} \nabla_1 h_{pq} - h_{11} \sigma_{k-1; i} \kappa_i^2 + k f  h_{11}^{2}.
\end{aligned}
\end{equation}
Since $\{h_{ij}(X_{0})\}$ is diagonal, as \eqref{add-5} we have, at $X_{0}$,
\[
\sigma_k^{ij, pq} = \left\{ \begin{aligned}
   \sigma_{k - 2; ip} (\kappa) & \;\;\mbox{ if }~ i=j, p=q, i \neq q, \\
   - \sigma_{k - 2; ip} (\kappa) & \;\;\mbox{ if }~ i=q, j=p, i \neq j, \\
   0 & \;\;\mbox{ otherwise.}
\end{aligned} \right.
\]
By the concavity of $\sigma_k^{1/k}$ in $\Gamma_k$, we have
\[
\sum_{i\neq j} \sigma_{k-2; ij} \nabla_1 h_{ii} \nabla_1 h_{jj}
  \leq \left(1- \frac{1}{k}\right) \frac{1}{f} (\sigma_{k-1;i} \nabla_1 h_{ii})^2 = \left(1-\frac{1}{k}\right) \frac{(\nabla_{1}f)^2}{f}.
\]
Thus, as \eqref{interior-6}, by Codazzi equation, we have
\begin{equation}
\label{m72}
\begin{aligned}
- \sigma_k^{ij, pq} \nabla_1 h_{ij} \nabla_1 h_{pq}
  = \,& \sum_{i \neq j} \sigma_{k-2; ij} (\nabla_1 h_{ij})^2 - \sum_{i\neq j} \sigma_{k-2; ij} \nabla_1 h_{ii} \nabla_1 h_{jj}\\
  \geq \,& \sum_{i \geq 2} 2\sigma_{k-2; i1} (\nabla_i h_{11})^2 - C \frac{(\nabla_{1}f)^2}{f}.
\end{aligned}
\end{equation}
Applying Lemma \ref{Blocki} to $\psi = f^{1/(k-1)}$, we get
\begin{equation}
\label{m73}
\frac{(\nabla_{1}f)^2}{f} \leq \frac{C}{d^2} f^{1- 1/(k-1)},
\end{equation}
where $d := \mathrm{dist} (x, \partial \Omega)$ and the constant $C$ depends only on
$\|f^{1/(k-1)}\|_{C^{1,1} (\ol \Omega \times [-\mu_0-1, \mu_0+1])}$, where $\mu_0 := \sup_{\ol \Omega} |u|$.
Similar to \eqref{super-1} we have
\begin{equation}
\label{add-6}
0 \leq \ol u - u \leq B d \mbox{ in } \Omega
\end{equation}
for some constant $B$ depending only on $\|\ol u - u\|_{C^1 (\ol \Omega)}$
and we may assume $h_{11} d^2 \geq 1$ for otherwise we are done. Combining \eqref{m71}-\eqref{add-6}, we have
\begin{equation}\label{m7}
  \frac{1}{h_{11}}\sigma_{k}^{ii}\nabla_{ii}h_{11}\geq \frac{2}{h_{11}}\sum_{i\geq 2}\sigma_{k-2;i1}(\nabla_i h_{11})^{2}-Cf^{1- 1/(k-1)} -\sigma_{k-1;i}\kappa_{i}^{2}+k f  h_{11}.
\end{equation}
Combining \eqref{p2}, \eqref{m1}, \eqref{m4} and \eqref{m7}, we obtain
\begin{equation}\label{main}
  \begin{aligned}
  0\geq \,& -C\frac{\alpha f}{\rho}
            -\alpha\sigma_k^{ii}\frac{(\nabla_{i}\rho)^{2}}{\rho^{2}}
            - (-\frac{v}{v-a}\sigma_k^{ii}\kappa_{i}^{2}+Cf^{1-1/(k-1)})
            + \sigma_k^{ii}\frac{(\nabla_{i}v)^{2}}{(v-a)^{2}}\\
        \,& +b\sum\sigma_{k}^{ii}
            -C b f
            +\frac{2}{h_{11}}\sum_{i\geq 2}\sigma_{k-2;i1} (\nabla_i h_{11})^{2}-Cf^{1- 1/(k-1)}-\sigma_{k}^{ii}\kappa_{i}^{2}\\
        \,& +k f  h_{11}-\sigma_{k}^{ii}\frac{(\nabla_{i}h_{11})^{2}}{h_{11}^{2}}\\
   \geq \,& \left(k h_{11} - \frac{C \alpha}{\rho} - Cb\right) f
            -\alpha\sigma_{k}^{ii}\frac{(\nabla_{i}\rho)^{2}}{\rho^{2}}
            +\frac{a}{v-a}\sigma_{k}^{ii}\kappa_{i}^{2}-Cf^{1-1/(k-1)}\\
          \,& + b\sum\sigma_{k}^{ii}
         +\frac{2}{h_{11}}\sum_{i\geq 2}\sigma_{k-2;i1}(\nabla_i h_{11})^{2}
            +\sigma_{k}^{ii}\frac{(\nabla_{i}v)^{2}}{(v-a)^{2}}
            -\sigma_{k}^{ii}\frac{(\nabla_{i}h_{11})^{2}}{h_{11}^{2}}\\
      \geq \,& -\alpha\sigma_{k}^{ii}\frac{(\nabla_{i}\rho)^{2}}{\rho^{2}}
            +\frac{a}{v-a}\sigma_{k}^{ii}\kappa_{i}^{2}-Cf^{1-1/(k-1)}\\
          \,& + b\sum\sigma_{k}^{ii}
         +\frac{2}{h_{11}}\sum_{i\geq 2}\sigma_{k-2;i1}(\nabla_i h_{11})^{2}
            +\sigma_{k}^{ii}\frac{(\nabla_{i}v)^{2}}{(v-a)^{2}}
            -\sigma_{k}^{ii}\frac{(\nabla_{i}h_{11})^{2}}{h_{11}^{2}}.
\end{aligned}
\end{equation}

As in Section 3, we consider two cases.
\\ \hspace*{\fill} \\
\noindent
\textbf{Case 1.} $h_{kk}\geq \varepsilon h_{11}$  for some $\varepsilon>0$ to be chosen.
It follows that
\begin{equation}\label{case-1}
  \sum \sigma_{k}^{ii}h_{ii}^{2}>\sigma_{k}^{kk}h_{kk}^{2}\geq \theta h_{11}^{2}\sum \sigma_{k}^{ii},
\end{equation}
where $\theta=\theta(n,k,\varepsilon)$.

By \eqref{p1} and Cauchy-Schwarz inequality, we have, for any $\epsilon_0 > 0$,
\begin{equation}\label{m8}
\begin{aligned}
  \frac{(\nabla_{i}h_{11})^{2}}{h_{11}^{2}}
      \,& = \Big(\alpha\frac{\nabla_{i}\rho}{\rho}+ \frac{\nabla_{i}v}{v-a}+b\langle X, e_{i} \rangle\Big)^{2}\\
      \,&\leq (1+\epsilon_0)\frac{(\nabla_{i}v)^{2}}{(v-a)^{2}}
         +   C\left(1+\frac{1}{\epsilon_0}\right)\left(\frac{(\nabla_{i}\rho)^{2}}{\rho^{2}}+b^{2}\right).
\end{aligned}
\end{equation}
By \eqref{m3}, \eqref{main}, \eqref{case-1} and \eqref{m8}, we find
\begin{equation}\label{main1}
  \begin{aligned}
  0  \geq \,& -\frac{C}{\epsilon_0\rho^{2}}\sum\sigma_{k}^{ii}
             +\left(\frac{a}{v-a}-\frac{\epsilon_0}{(v-a)^{2}}\right)\sigma_{k-1;i}\kappa_{i}^{2}\\
          \,& +(b-Cb^{2}/\epsilon_0)\sum\sigma_{k}^{ii}+\frac{2}{h_{11}}\sum_{i\geq 2}\sigma_{k-2;i1}(\nabla_i h_{11})^{2}-Cf^{1-1/(k-1)}.
\end{aligned}
\end{equation}
Now we fix $\epsilon_0$ sufficiently small such that
\[
\frac{a}{v-a}-\frac{\epsilon_0}{(v-a)^{2}} \geq \delta_0 > 0
\]
for some positive constant $\delta_0$. We see $b_0 := b-Cb^{2}/\epsilon_0 > 0$ if $b$ is small enough.
Thus, \eqref{main1} becomes
\begin{equation}\label{add-7}
  0 \geq \delta_0 \theta h_{11}^{2}\sum\sigma_{k}^{ii}
             -\frac{C}{\rho^{2}}\sum\sigma_{k}^{ii}
          +b_{0}\sum\sigma_{k}^{ii}-Cf^{1-1/(k-1)}.
\end{equation}
Using \eqref{nm} as \eqref{GC2-9}, we have
\begin{equation}
\label{m5}
\sum \sigma_{k}^{ii} = (n - k + 1) \sigma_{k-1} (\kappa)
                     \geq c_{0}h_{11}^{1/(k-1)}f^{1-1/(k-1)}.
\end{equation}
Then we get $\rho h_{11} \leq C$ from \eqref{add-7}.
\\ \hspace*{\fill} \\
\noindent
\textbf{Case 2.} $h_{kk}< \varepsilon h_{11}$.
Therefore we have $|h_{jj}|\leq C\varepsilon h_{11}$ for $j=k, \ldots, n$. By Lemma \ref{lem-3}, we fix
$\varepsilon$ such that
\begin{equation}\label{case-2}
  \frac{2}{h_{11}}\sum_{i\geq 2}\sigma_{k-2;i1}(\nabla_i h_{11})^{2}
             -(1+\frac{8}{3\alpha})\sum_{i\geq 2}\sigma_{k-1;i}\frac{(\nabla_{i}h_{11})^{2}}{h_{11}^{2}}\geq 0.
\end{equation}
By \eqref{p1} and Cauchy-Schwarz inequality, we see
\begin{equation}\label{m2}
\begin{aligned}
  \alpha^{2}\frac{(\nabla_{i}\rho)^{2}}{\rho^{2}}
           \,& =   \Big\{ \frac{\nabla_{i}v}{v-a}
               +   b\langle X, e_{i} \rangle+\frac{\nabla_{i}h_{11}}{h_{11}} \Big\}^{2}\\
           \,&\leq \frac{8}{3}\frac{(\nabla_{i}v)^{2}}{(v-a)^{2}}+Cb^{2}+ \frac{8}{3}\frac{(\nabla_{i}h_{11})^{2}}{h_{11}^{2}}.
\end{aligned}
\end{equation}
Using \eqref{p1}, \eqref{main}, \eqref{m8}, \eqref{case-2}, \eqref{m2}, \eqref{m5} and that $\alpha = \max\{3, k-1\}>\frac{8}{3}$, fixing $\epsilon_{0}$ sufficiently small, we have
\begin{equation}\label{main2}
  \begin{aligned}
  0 \geq  \,& \sum_{i\geq 2}\frac{a}{v-a}\sigma_{k}^{ii}\kappa_{i}^{2}+ (b-Cb^2)\sum\sigma_{k}^{ii}-Cf^{1-1/(k-1)}\\
         \,& +\frac{2}{h_{11}}\sum_{i\geq 2}\sigma_{k-2;i1}h_{i11}^{2}
             -(1+\frac{8}{3\alpha})\sum_{i\geq 2}\sigma_{k-1;i}\frac{(\nabla_{i}h_{11})^{2}}{h_{11}^{2}}\\
         \,& +\sigma_{k}^{11}(\frac{v}{v-a}\kappa_{1}^{2}-\epsilon_{0}C\kappa_{1}^{2}-\frac{C}{\rho^{2}}-C)\\
   \geq  \,& (b_{0}c_0 h_{11}^{1/(k-1)}-\frac{C}{\rho}-C)f^{1-1/(k-1)}
\end{aligned}
\end{equation}
provided $h_{11} \rho$ is sufficiently large, where $b_0 := b - Cb^2 > 0$ by fixing $b$ sufficiently small.
Thus, a bound $\rho^{\alpha}h_{11}\leq C$ follows by \eqref{main2}. Theorem \ref{interior} is proved.
\end{proof}

\begin{remark}
The exponent $\alpha$ can be improved to $\max \{\alpha_{0}, k-1\}$, where $\alpha_{0}$ is any constant strictly larger than 2.
In fact, \eqref{m2} can be improved as
$$
\begin{aligned}
  \alpha^{2}\frac{(\nabla_{i}\rho)^{2}}{\rho^{2}}
           \,& =   \Big\{ \frac{\nabla_{i}v}{v-a}
               +   b\langle X, e_{i} \rangle+\frac{\nabla_{i}h_{11}}{h_{11}} \Big\}^{2}\\
           \,&\leq 2(1+\epsilon)\frac{(\nabla_{i}v)^{2}}{(v-a)^{2}}+C(1+\frac{1}{\epsilon})b^{2}+ 2(1+\epsilon)\frac{(\nabla_{i}h_{11})^{2}}{h_{11}^{2}}
\end{aligned}
$$
for any $\epsilon > 0$ and the same conclusion follows with the constant $\alpha$ replaced by $\max \{\alpha_{0}, k-1\}$.
\end{remark}

\section{The Dirichlet problems}

In this section, we consider the existence of solutions to the Dirichlet problems \eqref{1-1} and \eqref{1-1h}.
In particular, we shall prove Theorem \ref{cor2}, \ref{cor1} and \ref{hessian}.

\begin{proof}[Proof of Theorem \ref{cor2}]
We consider the approximate problem as in \cite{JW21}. Note that there exists a positive constant $\theta_0$ satisfying
\[
\sigma_k (\kappa [M_{\ul u}]) \geq \theta_0^{k-1} \mbox{ on } \ol \Omega
\]
since $\kappa [M_{\ul u}] \in \Gamma_k$. Let $\eta$ be a cut-off function defined on $[0, \infty)$, satisfying $0 \leq \eta \leq 1$, $|\eta'| \leq C \theta_0^{-1}$, $|\eta''| \leq C \theta_0^{-2}$ and
\[
\eta (t) = \left\{ \begin{aligned}
   & 1  \;\;\mbox{ if }~ 0 \leq t \leq \frac{\theta_0}{4}, \\
                 & 0  \;\;\mbox{ if }~ \frac{\theta_0}{2} \leq t < \infty.
\end{aligned} \right.
\]
Consider the approximate problem
\begin{equation}
\label{final}
\left\{ \begin{aligned}
   \sigma_k \big(\kappa[M_u]\big) & = f_\epsilon (x, u) := [\tilde{f} (x, u) + \epsilon \eta (\tilde{f} (x, u))]^{k-1}  \;\;\mbox{ in }~ \Omega, \\
                 u &= \varphi  \;\;\mbox{ on }~ \partial \Omega,
\end{aligned} \right.
\end{equation}
where $\tilde{f}=f^{1/(k-1)}$. When $\epsilon$ is sufficiently small, by \eqref{subsol}, $\ul u$ is also a sub-solution of \eqref{final} and $(f_\epsilon)_u \geq 0$. By using the existence theorems in \cite{I91},  there exists a unique admissible solution $u_{\epsilon}$ of \eqref{final}.
By the maximum principle and that $\varphi$ is affine, we see
\[
\ul u \leq u_\epsilon \leq \varphi,
\]
which provides a uniform $C^0$ bound of $u_\epsilon$.
By \eqref{nm} with $l=s=0$ and $r=1$, we find
\[
H (M_u) = \sigma_1 (\kappa [M_u]) \geq c_0 f^{1/k}
\]
for some positive constant $c_0$ depending only on $n$ and $k$. We may assume $f \not \equiv 0$ for otherwise $u = \varphi$ is the solution of \eqref{1-1}.
Therefore, there exists a point $x_0 \in \Omega$ and a positive constant $\delta_0$ such that $f (x_0, u(x_0)) \geq \delta_0 > 0$.
We consider a neighbourhood $B_{\delta_1} (x_0)$ of $x_0$ such that $B_{\delta_1} (x_0) \subset \subset \Omega$.
By Theorem 5.4 of \cite{JW21}, we have
\[
\sup_{B_{\delta_1} (x_0)}|D u_\epsilon| \leq C,
\]
where the positive constant $C$ may depend on the distance from $\partial B_{\delta_1} (x_0)$ to $\Omega$.
Therefore, we can find a smaller neighbourhood $B_{\delta_2} (x_0)$ of $x_0$ such that
\[
\inf_{B_{\delta_2} (x_0)} f \geq \frac{\delta_0}{2}.
\]
Let $\overline{f}$ be a smooth function such that $\overline{f} = c_0 (\delta_0/2)^{1/k}$ in $B_{\delta_2/2} (x_0)$,
$\overline{f} \leq c_0 (\delta_0/2)^{1/k}$ in $B_{\delta_2} (x_0) - B_{\delta_2/2} (x_0)$ and $\overline{f} = 0$
outside $B_{\delta_2} (x_0)$.
Now by Theorem 16.10 of \cite{GT}, there exists a unique solution $v$ of the prescribed mean curvature equation
\[
\sigma_1 (\kappa[M_{v}]) = \epsilon_2 \overline{f} \mbox{ in } \Omega
\]
with $v = \varphi$ on $\partial \Omega$,  if the positive constant $\epsilon_2<1$ is chosen sufficiently small.
Note that $\sigma_1 (\kappa [M_\varphi]) = 0$ since $\varphi$ is affine. By the strong maximum principle, we see
\[
u_\epsilon - \varphi < v - \varphi < 0 \mbox{ in } \Omega
\]
when $\epsilon$ is small enough. Furthermore, for any $\Omega' \subset \subset \Omega$, there exists a positive constant $c_{\Omega'}$
depending only on $\Omega'$, $v$ and $\varphi$ such that
\[
u_\epsilon - \varphi < v - \varphi \leq - c_{\Omega'} \mbox{ on } \ol{\Omega'}.
\]
By Theorem 5.4 of \cite{JW21} again, we see
\[
\sup_{\ol {\Omega}'}|D u_\epsilon| \leq C_{\Omega'}.
\]
By the Pogorelov estimate \eqref{curvature} with $\ol u$ replaced by $\varphi$, letting $\epsilon \rightarrow 0$ (passing to a subsequence if necessary), we obtain a solution $u \in C^{1,1} (\Omega) \cap C^{0} (\ol \Omega)$. The solution is unique by
the maximum principle.
\end{proof}

\begin{proof}[Proof of Theorem \ref{cor1}]
Let $\{\Omega_j\}$ be a sequence of
uniformly $(k-1)$-convex domains with $\partial \Omega_j \in C^{1,1}$ such that $\Omega_j \uparrow \Omega$
as $j \uparrow \infty$. By \cite{CNS} (seeing \cite{T95} also), there exists a unique $k$-convex solution $u_j \in C^{3,1} (\ol{\Omega}_j)$
of the approximate equation
\begin{equation}
\label{1-1'}
\left\{ \begin{aligned}
   \sigma_k \big(\lambda(D^{2} u)\big) & = f+ 1/j  \;\;\mbox{ in }~ \Omega_j, \\
                 u &= \varphi  \;\;\mbox{ on }~ \partial \Omega_j
\end{aligned} \right.
\end{equation}
for each $j = 1, 2, \cdots$.
It is easy to construct a sub-solution $\ul u \in C^{1,1} (\ol \Omega)$ of \eqref{1-1h} by the uniform
$(k-1)$-convexity of $\Omega$ as in \cite{CNS} satisfying
\[
\left\{ \begin{aligned}
   \sigma_k (\lambda (D^2 \ul u)) & \geq f (x, \ul u) + \delta \;\;\mbox{ in }~ \Omega, \\
                 \ul u &= \varphi  \;\;\mbox{ on }~ \partial \Omega
\end{aligned} \right.
\]
for some $\delta > 0$ since $\ul u \leq \varphi$ and $f_u \geq 0$. By the maximum principle, we have
\begin{equation}
\label{add-8}
\ul u \leq u_j \leq \varphi \mbox{ in } \Omega_j \mbox{ for each } j = 1, 2, \cdots.
\end{equation}
As in the proof of Theorem \ref{cor2}, we consider any domain $\Omega' \subset \subset \Omega$.
By \eqref{add-8} and the interior gradient estimates established in Theorem 3.2 of \cite{CW01}, we see
\[
\|u_j\|_{C^1 (\ol{\Omega}')} \leq C_{\Omega'}, \mbox{ for } j \mbox{ sufficiently large},
\]
where the constant $C_{\Omega'}$ depends only on $n$, $k$ and $\mathrm{dist} (\Omega', \partial \Omega)$
but independent of $j$.
By the same arguments as in the proof of Theorem \ref{cor2}, we can also obtain a bound
\[
\|u_j\|_{C^2 (\ol{\Omega}')} \leq C_{\Omega'}, \mbox{ for } j \mbox{ sufficiently large}
\]
by \eqref{interior-1} with the function $\ol u$ replaced by $\varphi$. Letting $j \rightarrow \infty$ (passing to a
subsequence if necessary), we obtain a solution $u \in C^{1,1} (\Omega)$ of \eqref{1-1h}. By \eqref{add-8} again, we see
$u \in C^{1,1} (\Omega) \cap C^0 (\ol \Omega)$ and
\[
\lim_{x \rightarrow \partial \Omega} u = \varphi.
\]
Theorem \ref{cor1} is proved.
\end{proof}

Now we turn to consider the $k$-Hessian equation \eqref{1-1h} for general $\varphi$.
Since $f^{1/(k-1)} \in C^{1,1} (\ol \Omega)$ and $\Omega$ is uniformly $(k-1)$-convex, there exists a unique solution
$u_\epsilon \in C^{3,1} (\ol \Omega)$ of the approximate problem
\[
\left\{ \begin{aligned}
   \sigma_k (\lambda (D^2 u)) & \geq f (x) + \epsilon \;\;\mbox{ in }~ \Omega, \\
                 u &= \varphi  \;\;\mbox{ on }~ \partial \Omega.
\end{aligned} \right.
\]
By the $C^1$ estimates established in Section 2 of \cite{JW22}, we can obtain a viscosity solution $u \in C^{0,1} (\ol \Omega)$
for \eqref{1-1h} by letting $\epsilon \rightarrow 0$. Let $\ol u \in C^{1,1} (\ol \Omega)$ be the solution of
\[
\left\{ \begin{aligned}
   \sigma_k (\lambda (D^2 u)) & = 0 \;\;\mbox{ in }~ \Omega, \\
                 u &= \varphi  \;\;\mbox{ on }~ \partial \Omega.
\end{aligned} \right.
\]
\begin{lemma}
\label{add-9}
If $\{u < \ol u\} \neq \emptyset$,
we have $u \in C^{1,1} (\{u < \ol u\})$.
\end{lemma}
\begin{proof}
For any $x_0 \in \{u <\ol u\}$, there exists a neighborhood $D$ of $x_0$ satisfying $D \subset \ol D \subset \{u < \ol u\}$,
since $\{u < \ol u\}$ is open in $\Omega$. Note that $u_\epsilon$ (or a subsequence of $\{u_\epsilon\}$) converges to $u$ uniformly
on $\ol D$. Thus, there exists a positive constant $\delta$ such that $u_\epsilon - \ol u \leq - \delta$ on $\ol D$.
By the Pogorelov estimate \eqref{interior-1}, we get
\[
|D^2 u_\epsilon| \leq C \mbox{ on } \ol D
\]
for some positive constant $C$ independent of $\epsilon$. Letting $\epsilon \rightarrow 0$, we then obtain $u \in C^{1,1} (\ol D)$
and that $u \in C^{1,1} (\{u < \ol u\})$ follows immediately.
\end{proof}

\begin{lemma}\label{Blocki-prop}
Suppose \eqref{right} holds.
We have $u<\ol u$ on the domain consisting of $\{f>0\}\cap\Omega$ and those
connected components of $\{f=0\}\cap\Omega$, which are compact.
\end{lemma}
The proof Lemma \ref{Blocki-prop} follows almost the same arguments of Proposition 4.1 of \cite{Blocki03b}. The reader is referred to \cite{Blocki03b}
for more details.

Theorem \ref{hessian} follows by combining Lemma \ref{add-9} and Lemma \ref{Blocki-prop}.

\section{Asymptotic Plateau type problem}

In this section, we consider the asymptotic Plateau type problem in hyperbolic space $\mathbb{H}^{n+1}$.
We shall use the half space model for hyperbolic space
\[\mathbb{H}^{n+1} = \{ (x, x_{n+1}) \in \mathbb{R}^{n+1} \big\vert x_{n+1} > 0\}, \]
endowed with the metric
\[ d s^2 = \frac{1}{x_{n+1}^2} \sum_{i = 1}^{n+1} d x_i^2. \]

Suppose $\Omega \subset \mathbb{R}^n \times \{0\} \cong \mathbb{R}^n$ is a domain and
$\Gamma := \partial \Omega = \{\Gamma_1, \ldots, \Gamma_m\}$ is a disjoint collection of smooth closed $(n - 1)$ dimensional submanifolds at
\[ \partial_{\infty} \mathbb{H}^{n+1} := \mathbb{R}^n \times \{0\} \cong \mathbb{R}^n. \]
In this section,
$\tilde{g}$, $\tilde{h}$, $\tilde{\kappa}=\{\tilde{\kappa}_{1}, \ldots, \tilde{\kappa}_{n}\}$ and $\tilde{\nabla}$ will always denote the metric, second fundamental form, principal curvatures and Levi-Civita connections of the hypersurface
$M_{u} := \{(x, u(x)): x \in \Omega\}$ in $\mathbb{H}^{n+1}$.
We consider prescribed $k$-th Weingarten curvature equation
\begin{equation} \label{eq2-1}
	\left\{ \begin{aligned}
		\sigma_{k} ( \tilde{\kappa} [ M_{u} ] ) &= f(x, u) \quad & \mbox{in} \,\, \Omega, \\
		u  &=    0 \quad & \mbox{on} \,\, \Gamma
	\end{aligned} \right.
\end{equation}
in $\mathbb{H}^{n+1}$. Similar to the above sections, a function $u \in C^2 (\Omega)$ is called admissible if $\tilde{\kappa} [M_{u}] \in \Gamma_k$
in $\Omega$. In this section, we assume that there exists an admissible subsolution $\ul u \in C^4 (\Omega) \cap C^0 (\ol \Omega)$ satisfying
\begin{equation} \label{eq2-2}
	\left\{ \begin{aligned}
		\sigma_{k} ( \tilde{\kappa} [ M_{\ul u} ] ) & \geq f(x, \ul u) \quad & \mbox{in} \,\, \Omega, \\
		\ul u  & =    0 \quad & \mbox{on} \,\, \Gamma.
	\end{aligned} \right.
\end{equation}
For $\epsilon > 0$, let
\[
\Gamma_\epsilon := \{x \in \Omega: \ul u (x) = \epsilon\},\ \ \Omega_\epsilon := \{x \in \Omega: \ul u (x) > \epsilon\}.
\]
$\Gamma_\epsilon$ is also assumed to be a regular boundary of $\Omega_\epsilon$ when $\epsilon$ is sufficiently small, namely,
$\Gamma_\epsilon \in C^4$ is $(n-1)$-dimensional, and $|D \ul u| > 0$ on $\Gamma_\epsilon$. Sui \cite{Sui19} introduced the following
approximated problem
\begin{equation} \label{eq2-3}
	\left\{ \begin{aligned}
		\sigma_{k} ( \tilde{\kappa} [ M_{u} ] ) &= f(x, u) \quad & \mbox{in} \,\, \Omega_\epsilon, \\
		u  & =    \epsilon \quad & \mbox{on} \,\, \Gamma_\epsilon.
	\end{aligned} \right.
\end{equation}
The existence of smooth solutions to \eqref{eq2-3}
and its uniform $C^1$ estimates (independent of $\epsilon$) were derived by Sui-Sun \cite{SuiSun22a} so that
the existence of Lipschitz continuous solutions to \eqref{eq2-1} was proved.
Later, they demonstrated the existence of smooth solutions to \eqref{eq2-1} in some special cases
by establishing a Pogorelov type estimate in \cite{SuiSun22} as follows. Suppose $\ol u$ is
an admissible solution of the homogenous problem
\begin{equation}\label{homo}
	\left\{ \begin{aligned}
		\sigma_{k} ( \tilde{\kappa} [ M_{u} ] ) &= 0 \quad & \mbox{in} \,\, \Omega, \\
		u  &=    0 \quad & \mbox{on} \,\, \Gamma.
	\end{aligned} \right.
\end{equation}
Let $u_\epsilon$ be a solution of \eqref{eq2-3} and $\epsilon_0 > 0$ sufficiently small. Define
\[
\Omega_{\epsilon_0}^\epsilon = \{x \in \Omega_\epsilon: (\ol u^2 - u_\epsilon^2) (x) > c\}, \mbox{ for }
   \epsilon < \frac{\epsilon_0}{2}.
\]
It was shown in \cite{SuiSun22} that there exists a positive constant $c$ depending on $\epsilon_{0}$ but independent of $\epsilon$
such that
$\ol u^2-u_\epsilon^2-c>0$ in $\Omega_{\epsilon_{0}}^{\epsilon}$ and $\ol u^2-u_\epsilon^2-c=0$ on $\partial\Omega_{\epsilon_{0}}^{\epsilon}$.
The following theorem contains the main results of Sui-Sun \cite{SuiSun22}.

\begin{theorem}  \label{Theorem6}
		(\cite{SuiSun22}) Suppose $\Omega = \{ x \in \mathbb{R}^n: |x| < R \}$ and $0< f(x,u) \in C^{2}(\Omega\times \mathbb{R})$.
		Then the second fundamental form $\tilde{h}_\epsilon$ of the graphic hypersurface $M_{u_\epsilon}$ satisfies the
		estimates
		\begin{equation}\label{eq2-4}
			(\ol u^2 - u_\epsilon^2-c)^{\alpha} |\tilde{h}_\epsilon| \leq C,
		\end{equation}
		in $\Omega_{\epsilon_{0}}^{\epsilon}$, where $\ol u=\sqrt{R^2-|x|^2}$, $\alpha$ and $C$ are positive constants depending on $||u_\epsilon||_{C^{1}(\overline{\Omega})}$ and other known data, but independent of $\epsilon$.
\end{theorem}
They needed the condition that the domain $\Omega$ is a ball because their methods require that the function $\ol u^2+|x|^2$ should be affine.
Our technique can improve their results as follows.
\begin{theorem}  \label{Theorem6'}
	Suppose $\Omega$ is bounded and $0< f^{1/(k-1)}(x,u) \in C^{2}(\Omega\times \mathbb{R})$.	
	Additionally, assume that there exists an admissible solution $\ol u$ of \eqref{homo},
	such that $\ol u^2+|x|^2$ is $(k+1)$-convex for $k < n$ or convex for $k=n$.	
	Then the estimate \eqref{eq2-4} holds and furthermore, the constants in \eqref{eq2-4} are independent of the lower bound of $f$.
\end{theorem}
Combining the solvability of \eqref{eq2-3} and the uniform $C^1$ estimates established in \cite{SuiSun22a}, we can improve Theorem 1.2 in \cite{SuiSun22}:
\begin{theorem}  \label{Theorem1.8}
	Let  $k < n$ and  $\Omega$ is bounded. Assume that $0 < f(x, u) \in C^{\infty} (\mathbb{H}^{n + 1})$ satisfies
\begin{equation}
\label{eq2-5}
f_u - \frac{f}{u} \geq 0,
\end{equation}
there exists an admissible function $\ul u \in C^4 (\Omega) \cap C^0 (\ol \Omega)$ satisfying \eqref{eq2-2} and
\begin{equation}
\label{eq2-6}
- \lambda (D^2 \ul u) \in \Gamma_{k+1} \mbox{ near } \Gamma,
\end{equation}
and the compatibility conditions (see the introduction of \cite{SuiSun22a}) hold.
In addition, suppose that there exists an admissible solution $\ol u$ of homogeneous problem \ref{homo} such that $\ol u^2+|x|^2$ is $(k+1)$-convex.
Then there exists a unique admissible solution $u \in C^{\infty}(\Omega) \cap C^0(\overline{\Omega})$ with $\ol u\geq u \geq \underline{u}$ to the asymptotic problem \eqref{eq2-1}.
\end{theorem}

\begin{example}
	Let
\[
\Omega = \{x = (x_1, \ldots, x_n) \in \mathbb{R}^n: \sum_{i=1}^{n-k+1}x_{i}^{2} + \frac{1}{2}\sum_{j=n-k+2}^{n}x_{j}^{2} < R^2\}
\]
be an ellipsoid and $\ol u=\sqrt{R^{2}-\sum_{i=1}^{n-k+1}x_{i}^{2}-\frac{1}{2}\sum_{j=n-k+2}^{n}x_{j}^{2}}$. It is easy to check that $\ol u^2+|x|^2$ is $(k+1)$-convex and that $\ol u$ is an admissible solution of homogeneous problem
	\[
	\left\{ \begin{aligned}
		\sigma_{k} ( \tilde{\kappa} [ M_{u} ] ) &= 0 \quad & \mbox{in} \,\, \Omega, \\
		u  &=    0 \quad & \mbox{on} \,\, \Gamma.
	\end{aligned} \right.
	\]	
\end{example}

In order to prove Theorem \ref{Theorem6'}, we first introduced some calculations on $M_u$. Most of their proofs can be found in \cite{Sui19}.

Let $\epsilon_{1},\ldots,\epsilon_{n+1}$ be the natural basis of $\mathbb{R}^{n+1}$.
When $M_u$ is viewed as a hypersurface in $\mathbb{H}^{n + 1}$, its unit upward normal, metric, second fundamental form are given by
\[ {\bf n} = \frac{u(-Du, 1)}{w}, \quad  \tilde{g}_{ij} = \frac{1}{u^2} ( \delta_{ij} + u_i u_j ), \quad
\tilde{h}_{ij} = \frac{1}{u^2 w} ( \delta_{ij} + u_i u_j + u u_{ij} ). \]
respectively, where $w := \sqrt{1 + |Du|^2}$.
Therefore,
the hyperbolic principal curvatures $\tilde{\kappa} [M_{u}]$ are the eigenvalues of the symmetric matrix $A [u] = (a_{ij})$, where
\[ a_{ij}
= \frac{1}{w} \gamma^{ik} ( \delta_{kl} + u_k u_l + u u_{kl} ) \gamma^{lj} =  \frac{1}{w} ( \delta_{ij} + u \gamma^{ik} u_{kl} \gamma^{lj} ), \]
where $\gamma^{ik}=\delta_{ik}-\frac{u_iu_k}{w(1+w)}$. Let $g$, $\nabla$ and $h$ denote the metric and Levi-Civita connection
and the second fundamental form on $M_u$
induced from $\mathbb{R}^{n+1}$ respectively.
We have
\begin{equation} \label{eq1-16}
	\tilde{h}_{ij} = \frac{1}{u} h_{ij} + \frac{v}{u^2} g_{ij},
\end{equation}
where $v = \nu \cdot \partial_{n + 1}$ and $\cdot$ is the inner product in $\mathbb{R}^{n+1}$, and
%
\begin{equation} \label{eq2-10}
	\tilde{\nabla}_{ij} w 
	= \nabla_{ij} w + \frac{1}{u}( \nabla_{i}u \nabla_{j}w + \nabla_{j}u \nabla_{i}w - g^{kl} \nabla_{l}u \nabla_{k}w g_{ij}),
\end{equation}
where $\tilde{\nabla}$ is the Levi-Civita connection on $M_u$ introduced from $\mathbb{H}^{n+1}$.
The following two lemmas were proved in \cite{Sui19}.
\begin{lemma}  \label{Lemma2-2}
	(\cite{Sui19}) When $M_{u}$ is viewed as a hypersurface in $\mathbb{R}^{n+1}$, we have
	\[ \tilde{\nabla}_{i}u = \tau_i \cdot \epsilon_{n + 1}, \quad \tilde{\nabla}_{i}(x^{\alpha}) = \tau_i \cdot \epsilon_{\alpha},  \quad \alpha = 1, \ldots, n, \]
	
	\[   g^{kl} \tilde{\nabla}_{k}u \tilde{\nabla}_{l}u  = |\nabla u|^2 = 1 - v^2, \]
	
	\[  \nabla_{ij} u = h_{ij} v, \quad \nabla_{ij} x^{\alpha} = h_{ij} \nu^{\alpha}, \quad \alpha = 1, \ldots, n, \]
	
	\[  \tilde{\nabla}_{i}v = - h_{ij} \,g^{j k} \tilde{\nabla}_{k}u,  \]
	and
	\[   \nabla_{ij} v = - g^{kl} ( v h_{il} h_{kj} + \tilde{\nabla}_{l}u \nabla_k h_{ij} ),  \]
	where $\cdot$ is the Euclidean inner product, $\epsilon_{1}, \ldots, \epsilon_{n+1}$ is natural basis of $\mathbb{R}^{n+1}$, $\tau_1, \ldots, \tau_n$ is any local frame on $M_{u}$, $X = (x_1, \ldots, x_n, u)$ is the position vector field on $M_{u}$, $\nu = (\nu_1, \ldots, \nu_{n + 1})$ is the unit normal vector field to $M_{u}$ in $\mathbb{R}^{n+1}$ and $v := \nu_{n+1} = \frac{1}{w}$.
\end{lemma}


\begin{lemma}  \label{Lemma2-1}
	(\cite{Sui19}) Let $M_u$ be an admissible hypersurface in $\mathbb{H}^{n+1}$ satisfying
	$\sigma_{k} (\kappa [M_u]) = f$.
	Then in a local orthonormal frame on $\Sigma$, for $v = \nu_{n+1} = \frac{1}{w}$, we have
	\[
	\begin{aligned}
		\sigma_{k}^{ij} \tilde{\nabla}_{ij} v
		= & - v \sigma_{k}^{ij} \tilde{h}_{ik} \tilde{h}_{kj} + \big( 1 + v^2 \big) \sigma_{k}^{ij} \tilde{h}_{ij} - v \sum \sigma_{k}^{ii} \\
		& - \frac{2}{u^2} \sigma_{k}^{ij} \tilde{h}_{jk} \tilde{\nabla}_{i}u \tilde{\nabla}_{k}u
		 + \frac{2 v}{u^2} \sigma_{k}^{ij} \tilde{\nabla}_{i}u \tilde{\nabla}_{j}u - \frac{\tilde{\nabla}_{k}u}{u} \tilde{\nabla}_{k}f.
	\end{aligned}
	\]
\end{lemma}

Similar to Lemma \ref{lem-4}, we have the following lemma.
\begin{lemma}
	\label{lem-4'}
	Suppose the function $\ol u^2+|x|^2 \in C^2 (\Omega)$ is $(k+1)$-convex
	, $u_\epsilon$ is a admissible solution of \eqref{eq2-3} and $c$ is a constant. Let $\rho=\ol u^2-u_\epsilon^2-c$,
and $\rho(X)=\rho(x,u):=\rho(x)$. We have
	\begin{equation}
		\label{add-1'}
		\lambda\Big(\tilde{\nabla}_{ij}\rho - \frac{\tilde{\nabla}_{i}u\tilde{\nabla}_{j}\rho+\tilde{\nabla}_{j}u\tilde{\nabla}_{i}\rho}{u} + (2u^2 w - \boldsymbol{n}(\rho)) \tilde{h}_{ij}\Big) \in \ol{\Gamma}_k
	\end{equation}
    	where $e_1, \ldots, e_n$ is a local orthonormal frame, $\tilde{h}$ is the second fundamental form and $\boldsymbol{n}$ is the unit normal vector field to $M_{u}$ in $\mathbb{H}^{n+1}$.
\end{lemma}
\begin{proof}
Let ${\bf D}$ denote the Levi-Civita connection in $\mathbb{H}^{n+1}$. The corresponding Christoffel symbols are given by
\[	{\bf\Gamma}_{ij}^k = \,\frac{1}{X_{n + 1}} \big(- \delta_{ik} \delta_{n + 1\, j} - \delta_{kj} \delta_{n + 1 \,i} + \delta_{k\, n + 1} \delta_{ij} \big).\]
Let $\{\epsilon_1, \ldots, \epsilon_{n+1}\}$ be the natural basis of $\mathbb{R}^{n+1}$
and $\{e_1,\ldots,e_n\}$ be orthonormal frame on $TM_u$ given by
\[
e_i := u\sum_{j=1}^n \gamma^{ij} (\epsilon_j + u_j \epsilon_{n+1}), \ \ i = 1, \ldots, n,
\]
where $\gamma^{ij}=\delta_{ij}-\frac{u_iu_j}{w(1+w)}$. Let $\tau$ be the matrix $\{\gamma^{ij}\}$.
Similar to Lemma \ref{lem-4}, by Lemma 3.1 of \cite{I91} or Proposition 2.1 of \cite{JW21}, we have
$\lambda (\tau D^2 (\ol u^2+|x|^2) \tau) \in \ol{\Gamma}_k$ and furthermore
\[
	\sigma_j (\lambda (\tau D^2 (\ol u^2+|x|^2) \tau)) \geq \frac{1}{w^2} \sigma_j (\lambda (D^2 (\ol u^2+|x|^2)))\geq 0, \ \ j = 1, \ldots, k.
\]
As in the proof of Lemma \ref{lem-4}, we regard $\ol u^2+|x|^2$ as a function in a neighborhood of $M_u$ in $\mathbb{H}^{n+1}$ and we have
\[
\begin{aligned}
		\tilde{\nabla}_{ij} \rho
		= \,& \bold{D}_{ij}\rho+\tilde{h}_{ij}\boldsymbol{n}(\rho)\\
		= \,& u^{2}\gamma^{is}\gamma^{tj}\bold{D}^2 \rho (\epsilon_s+u_{s}\epsilon_{n+1}, \epsilon_t+u_{t}\epsilon_{n+1})+\tilde{h}_{ij}\boldsymbol{n}(\rho)\\
		= \,& u^{2}\gamma^{is} \gamma^{jt}\Big(\bold{D}^2 \rho (\epsilon_s, \epsilon_t) + u_{t}\bold{D}^2 \rho (\epsilon_s, \epsilon_{n+1}) + u_{s}\bold{D}^2 \rho (\epsilon_{n+1}, \epsilon_{t})\\
		    & +u_{s}u_{t}\bold{D}^2 \rho(\epsilon_{n+1}, \epsilon_{n+1})\Big) + \tilde{h}_{ij}\boldsymbol{n}(\rho)\\
		= \,& u^{2}\gamma^{is} \gamma^{jt}\Big((\ol u^2-u^2)_{st}-u_{t}\boldsymbol{\Gamma}_{s\;n+1}^{\alpha}\rho_{\alpha}
		-u_{s}\boldsymbol{\Gamma}_{n+1\;t}^{\alpha}\rho_{\alpha}\Big) + \tilde{h}_{ij}\boldsymbol{n}(\rho)\\
		 = \,& u^{2}\gamma^{is} \gamma^{jt}\Big((\ol u^2-u^2)_{st} + \frac{u_{t}\rho_{s}+u_{s}\rho_{t}}{u}
		  \Big)+ \tilde{h}_{ij}\boldsymbol{n}(\rho)\\
		 = \,& u^{2}\gamma^{is} \gamma^{jt}(\ol u^2+|x|^2)_{st} +\frac{\tilde{\nabla}_{i}u\tilde{\nabla}_{j}\rho
		 	+\tilde{\nabla}_{j}u\tilde{\nabla}_{i}\rho}{u}- u^{2}\gamma^{is} \gamma^{jt}(u^2+|x|^2)_{st}+ \tilde{h}_{ij}\boldsymbol{n}(\rho)\\
		 = \,& u^{2}\gamma^{is} \gamma^{jt}(\ol u^2+|x|^2)_{st} + \frac{\tilde{\nabla}_{i}u\tilde{\nabla}_{j}\rho+\tilde{\nabla}_{j}u\tilde{\nabla}_{i}\rho}{u}-2 u^2 w \tilde{h}_{ij}+ \tilde{h}_{ij}\boldsymbol{n}(\rho).
	\end{aligned}
\]
Therefore, \eqref{add-1'} is proved.
\end{proof}

We are ready to prove Theorem \ref{Theorem6'}.

\begin{proof}[Proof of Theorem \ref{Theorem6'}] We will drop the subscript $\epsilon$ for convenience since there is no possible confusion.
Let $\cdot$ be the Euclidean inner product and $ v := 1/w = \nu \cdot \epsilon_{n+1}$. There exists a positive constant $a$ depending only on
$\|u\|_{C^{1}(\Omega_{\epsilon_0}^\epsilon)}$ and $\epsilon_0$ but independent of $\epsilon$ such that $v\geq 2a$ on $\Omega_{\epsilon_0}^\epsilon$.
We consider the test function
\[
	W =
	\rho^\alpha (v - a)^{-1} \exp\{\frac{\beta}{u}\}\tilde{h}_{\xi\xi},
\]
where $\rho=\ol u^2 - u^2 - c$, $X\in M_u$, $\xi \in T_X M_u$ is a unit vector and $\alpha \geq 1$, $\beta > 0$ are positive
constants to be determined.
Suppose that the maximum value of $W$ on $\Omega_{\epsilon_0}^\epsilon$ is achieved at an interior point $X_{0} = (x_{0}, u(x_{0}))\in M_u$, $x_{0}\in\Omega_{\epsilon_0}^\epsilon$
and $\xi_0 \in T_{X_0} M_u$.
Let $\{e_{1}, e_{2}, \ldots, e_{n}\}$
be a local orthonormal frame with respect to hyperbolic metric
about $X_{0}$. We may assume $\xi_0 = e_1$ and
$\{\tilde{h}_{ij}\}$ is diagonal at $X_{0}$ with $\tilde{h}_{11} \geq \cdots \geq \tilde{h}_{nn}$ by a rotation if necessary.
Let $\tilde{\kappa}_i = \tilde{h}_{ii} (X_0)$ for $i = 1, \ldots, n$.
Therefore, at $X_{0}$ , taking the covariant derivatives twice with respect
to
\[ \log \tilde{h}_{11} - \log ( v - a ) + \frac{\beta}{u} + \alpha \log \rho, \]
 we have
\begin{equation} \label{eq2-16}
	\frac{\tilde{\nabla}_{i}\tilde{h}_{11}}{\tilde{h}_{11}} - \frac{\tilde{\nabla}_i v}{v - a} - \beta \frac{\tilde{\nabla}_{i}u}{u^2} + \alpha \frac{\tilde{\nabla}_{i}\rho}{\rho} = 0,
\end{equation}
\begin{equation} \label{eq2-17}
	\begin{aligned}
		\sigma_k^{ii} &\Big(\frac{\tilde{\nabla}_{ii}\tilde{h}_{11}}{\tilde{h}_{11}} - \frac{ (\tilde{\nabla}_{i}\tilde{h}_{11})^2}{\tilde{h}_{11}^2} - \frac{ \tilde{\nabla}_{ii} v}{v - a}
		+ \frac{ (\tilde{\nabla}_{i}v)^2}{(v - a)^2} \\
		&- \beta \frac{ \tilde{\nabla}_{ii} u}{u^2} + \beta  \frac{2 (\tilde{\nabla}_{i}u)^2}{u^3} + \alpha \frac{ \tilde{\nabla}_{ii} \rho}{\rho} - \alpha  \frac{(\tilde{\nabla}_{i}\rho)^2}{ \rho^2}\Big) \leq 0.
	\end{aligned}
\end{equation}

By \eqref{eq2-10}, Lemma \ref{Lemma2-2} and Lemma \ref{Lemma2-1}, we have
\begin{equation} \label{eq2-24}
	\sigma_{k}^{ii} \tilde{\nabla}_{ii} u =  ku vf + \frac{2}{u} \sigma_{k}^{ii} (\tilde{\nabla}_{i}u)^2 - u \sum \sigma_{k}^{ii},
\end{equation}

By Lemma \ref{lem-4'} and Lemma \ref{Lemma2-1}, we have
\begin{equation}\label{eq2-11}
	\begin{aligned}
		\sigma_k^{ii} \tilde{\nabla}_{ii} \rho
		\geq \,& -(2u^2 w - \boldsymbol{n}(\rho))\sigma_k^{ii}\tilde{h}_{ij}+ 2\sigma_{k}^{ii}\frac{\tilde{\nabla}_{i}u}{u}\tilde{\nabla}_{i}\rho\\
		\geq \,& -Cf+2\sigma_{k}^{ii}\frac{\tilde{\nabla}_{i}u}{u}\tilde{\nabla}_{i}\rho
	\end{aligned}
\end{equation}
and
\begin{equation} \label{eq2-13}
	\begin{aligned}
		\sigma_k^{ii} \tilde{\nabla}_{ii}v
		= & - v \sigma_k^{ii} \tilde{\kappa}_i^2 + k\big( 1 + v^2 \big) f - v \sum \sigma_k^{ii} \\
		& -   \frac{2 \sigma_k^{ii}\tilde{\kappa}_i (\tilde{\nabla}_{i}u)^2}{u^2} + \frac{2v \sigma_k^{ii} (\tilde{\nabla}_{i}u)^2}{u^2}  - \frac{\tilde{\nabla}_{i}u}{u} \tilde{\nabla}_{k}f.
	\end{aligned}
\end{equation}
Differentiating equation \eqref{eq2-1} twice, we obtain
\begin{equation}
	\label{d1-2}
	|\sigma_k^{ii}  \tilde{\nabla}_{l} \tilde{h}_{ii}| = |\tilde{\nabla}_{l}f| \leq C f^{1 - 1/(k-1)}, \mbox{ for each } l = 1, \ldots, n
\end{equation}
and
\begin{equation}
	\label{eq2-9'}
	\sigma_k^{ii} \tilde{\nabla}_{11} \tilde{h}_{ii} = \tilde{\nabla}_{11}f - \sigma_k^{ij, pq} \tilde{\nabla}_{1}\tilde{h}_{ij}\tilde{\nabla}_{1} \tilde{h}_{pq},
\end{equation}
where
\[
\sigma_k^{ij, pq} := \frac{\partial^2 \sigma_k (\lambda (h))}{\partial \tilde{h}_{ij} \partial \tilde{h}_{pq}}.
\]
By direct calculations, we have
\[
\tilde{\nabla}_{11}f\geq -C\frac{(\tilde{\nabla}_{1}f)^{2}}{f}-C\tilde{\kappa}_{1}f^{1- 1/(k-1)},
\]
provided $\tilde{\nabla}_{11}$ is sufficiently large.

Applying Lemma \ref{Blocki} to $\psi = f^{1/(k-1)}$ as before, we get
\[
	\frac{(\tilde{\nabla}_{1}f)^2}{f} \leq \frac{C}{d^2} f^{1- 1/(k-1)},
\]
where $d := \mathrm{dist} (x, \partial \Omega_{\epsilon_0}^\epsilon)$ and the constant $C$ depends only on
$\|f^{1/(k-1)}\|_{C^{1,1} (\ol \Omega_{\epsilon_0}^\epsilon \times [-\mu_0-1, \mu_0+1])}$, where $\mu_0 := \sup_{\ol \Omega_{\epsilon_0}^\epsilon} |u|$.
Since
\[
\frac{|\ol u^2 - u^2- c|}{|x-x_{0}|} \leq \sup_{\Omega_{\epsilon_0}^\epsilon}|D(\ol u^2 - u^2- c)|
\]
for $x_{0}\in\partial\Omega_{\epsilon_{0}}^{\epsilon}$. Thus,
similar to \eqref{super-1}, we have
\[
	\label{add-6'}
	0 \leq \ol u^2 - u^2- c\leq B d \mbox{ in } \Omega_{\epsilon_{0}}^{\epsilon},
\]
where $B := \sup_{\Omega_{\epsilon_0}^\epsilon}|D(\ol u^2 - u^2- c)|$.
Consequently, we may assume $\tilde{\kappa}_{1} d^2 \geq 1$ for otherwise we are done.
Combining with \eqref{eq2-9'} we find
\begin{equation}
	\label{eq2-9}
	\frac{1}{\tilde{\kappa}_{1}}\sigma_k^{ii} \tilde{\nabla}_{11} \tilde{h}_{ii} \geq -Cf^{1- 1/(k-1)} -\frac{1}{\tilde{\kappa}_{1}} \sigma_k^{ij, pq} \tilde{\nabla}_{1}\tilde{h}_{ij}\tilde{\nabla}_{1} \tilde{h}_{pq}.
\end{equation}
Noting that $\{\tilde{h}_{ij}(X_{0})\}$ is diagonal, as before, we have, at $X_{0}$,
\[
\sigma_k^{ij, pq} = \left\{ \begin{aligned}
	\sigma_{k - 2; ip} (\tilde{\kappa}) & \;\;\mbox{ if }~ i=j, p=q, i \neq q, \\
	- \sigma_{k - 2; ip} (\tilde{\kappa}) & \;\;\mbox{ if }~ i=q, j=p, i \neq j, \\
	0 & \;\;\mbox{ otherwise.}
\end{aligned} \right.
\]
By the concavity of $\sigma_k^{1/k}$ in $\Gamma_k$, we have
\[
\sum_{i\neq j} \sigma_{k-2; ij} \tilde{\nabla}_1 \tilde{h}_{ii} \tilde{\nabla}_1 \tilde{h}_{jj}
\leq \left(1- \frac{1}{k}\right) \frac{1}{f} (\sigma_{k-1;i} \tilde{\nabla}_1 \tilde{h}_{ii})^2 = \left(1-\frac{1}{k}\right) \frac{(\tilde{\nabla}_{1}f)^2}{f}.
\]
Thus, by Codazzi equation we have
\begin{equation}
	\label{m72-h}
	\begin{aligned}
		- \sigma_k^{ij, pq} \tilde{\nabla}_1 \tilde{h}_{ij} \tilde{\nabla}_1 \tilde{h}_{pq}
		= \,& \sum_{i \neq j} \sigma_{k-2; ij} (\tilde{\nabla}_1 \tilde{h}_{ij})^2 - \sum_{i\neq j} \sigma_{k-2; ij} \tilde{\nabla}_1 \tilde{h}_{ii} \tilde{\nabla}_1 \tilde{h}_{jj}\\
		\geq \,& \sum_{i \neq j} \sigma_{k-2; ij} (\tilde{\nabla}_1 \tilde{h}_{ij})^2 - C \frac{(\tilde{\nabla}_{1}f)^2}{f}\\
		\geq \,& \sum_{i \geq 2} 2\sigma_{k-2; i1} (\tilde{\nabla}_i \tilde{h}_{11})^2 - C \tilde{\kappa}_{1}f^{1- 1/(k-1)}.
	\end{aligned}
\end{equation}
Since $\mathbb{H}^{n+1}$ has constant sectional curvature $-1$, by the Codazzi and Gauss
equations we have $\tilde{\nabla}_{k}h_{ij} = \tilde{\nabla}_{j}h_{ik}$ and
\begin{equation}\label{changeorder}
	\tilde{\nabla}_{11}\tilde{h}_{ii}=\tilde{\nabla}_{ii}\tilde{h}_{11}+(\tilde{\kappa}_{i}\tilde{\kappa}_{1}-1)(\tilde{\kappa}_{i}-\tilde{\kappa}_{1})
\end{equation}
at $X_0$.
Taking \eqref{changeorder}, \eqref{eq2-24}--\eqref{m72-h} into \eqref{eq2-17} yields,
\begin{equation} \label{eq2-18}
	\begin{aligned}
		& \Big( \tilde{\kappa}_1 - \frac{1 + v^2}{v - a} - \frac{\beta v}{u} - \frac{\alpha C}{\rho} \Big)kf  - C f^{1- 1/(k-1)} \\
		& + \Big( \frac{\beta}{u} + \frac{a}{v - a} -C\Big) \sum \sigma_{k}^{ii}
		+ \frac{v}{v - a} \sigma_{k}^{ii} \tilde{\kappa}_i^2  \\
		& + \frac{2}{v - a}   \sigma_{k}^{ii} \tilde{\kappa}_i \frac{(\tilde{\nabla}_{i}u)^2}{u^2}
		- \frac{2 v}{v - a}  \sigma_{k}^{ii} \frac{(\tilde{\nabla}_{i}u)^2}{u^2} - \alpha  \sigma_{k}^{ii} \frac{(\tilde{\nabla}_{i}\rho)^2}{\rho^2}\\
		&  +\frac{2 \alpha \sigma_{k}^{ii} \tilde{\nabla}_{i}u \tilde{\nabla}_{i}\rho}{u\rho} + \frac{\sigma_{k}^{ii} (\tilde{\nabla}_{i}v)^2}{(v - a)^2}
		+ \frac{2\sum_{i \geq 2} \sigma_{k-2; i1} (\tilde{\nabla}_i \tilde{h}_{11})^2 }{\tilde{\kappa}_1}\\
		& - \frac{\sigma_{k}^{ii} (\tilde{\nabla}_{i}\tilde{h}_{11})^2}{\tilde{\kappa}_1^2}  \leq 0.
	\end{aligned}
\end{equation}
By Lemma \ref{Lemma2-2} and Cauchy-Schwartz inequality, we have
\begin{equation*}
	\begin{aligned}
		& \frac{2}{v - a}  \sigma_{k}^{ii} \tilde{\kappa}_i \frac{(\tilde{\nabla}_{i}u)^2}{u^2} -  \frac{2 v}{v - a}  \sigma_{k}^{ii} \frac{(\tilde{\nabla}_{i}u)^2}{u^2} \\
		\geq & - \frac{2}{v - a}  \sigma_{k}^{ii} |\tilde{\kappa}_i| - \frac{2}{v - a}  \sum \sigma_{k}^{ii} \\
		\geq & -  \frac{4}{ a (v - a)}  \sum \sigma_{k}^{ii}  - \frac{a}{4 (v - a)} \sigma_{k}^{ii} \tilde{\kappa}_i^2 - \frac{2}{v - a} \sum \sigma_{k}^{ii}.
	\end{aligned}
\end{equation*}
Hence, \eqref{eq2-18} reduces to
\begin{equation} \label{eq2-18-1}
	\begin{aligned}
		& \Big( \tilde{\kappa}_1 - \frac{1 + v^2}{v - a} - \frac{\beta v}{u} - \frac{\alpha C}{\rho} \Big) kf- C f^{1-1/(k-1)} \\
		& + \Big( \frac{\beta}{u} + \frac{a - 2 - 4 a^{- 1}}{v - a}-C \Big)\sum \sigma_{k}^{ii}
		+ \frac{3 a}{4(v - a)} \sigma_{k}^{ii} \tilde{\kappa}_i^2  \\
		& - \alpha \sigma_{k}^{ii} \frac{(\tilde{\nabla}_{i}\rho)^2}{\rho^2} +\frac{2 \alpha \sigma_{k}^{ii} \tilde{\nabla}_{i}u \tilde{\nabla}_{i}\rho}{u \rho}
		+ \frac{\sigma_{k}^{ii}(\tilde{\nabla}_{i}v)^2}{(v - a)^2} \\
		& + \frac{2\sum_{i \geq 2} \sigma_{k-2; i1} (\tilde{\nabla}_i \tilde{h}_{11})^2 }{\tilde{\kappa}_1}- \frac{\sigma_{k}^{ii} (\tilde{\nabla}_{i}\tilde{h}_{11})^2}{\tilde{\kappa}_1^2}  \leq 0.
	\end{aligned}
\end{equation}
We consider two cases as before.

\noindent
\textbf{Case 1.} $\tilde{h}_{kk}\geq \varepsilon \tilde{h}_{11}$  for some $\varepsilon>0$ to be chosen.
We have
\[
	\sum \sigma_{k}^{ii}\tilde{h}_{ii}^{2}>\sigma_{k}^{kk}\tilde{h}_{kk}^{2}\geq \theta \tilde{h}_{11}^{2}\sum \sigma_{k}^{ii}
\]
as \eqref{case-1} in Section 4.

By \eqref{eq2-16} and Cauchy-Schwartz inequality, we have
\begin{equation} \label{eq2-20}
	\begin{aligned}
		& \frac{\tilde{\nabla}_{i}\sigma_{k}^{ii}(\tilde{\nabla}_{i}\tilde{h}_{11})^2}{\tilde{\kappa}_1^2} \leq (1 + \delta_1) \sigma_{k}^{ii} \frac{(\tilde{\nabla}_{i}v)^2}{(v - a)^2} \\
		& + 2 (1 + \delta_1^{- 1}) \beta^2 \sigma_{k}^{ii} \frac{(\tilde{\nabla}_{i}u)^2}{u^4} + 2 (1 + \delta_1^{- 1}) \alpha^2 \sigma_{k}^{ii} \frac{(\tilde{\nabla}_{i}\rho)^2}{\rho^2},
	\end{aligned}
\end{equation}
where $\delta_1$ is a positive constant to be determined later.
Also, we have
\begin{equation} \label{eq2-21}
	\frac{2 \alpha \sigma_{k}^{ii} \tilde{\nabla}_{i}u \tilde{\nabla}_{i}\rho}{u \rho}
	\geq  - \sum \sigma_{k}^{ii} - \alpha^2 \sigma_{k}^{ii} \frac{ (\tilde{\nabla}_{i}\rho)^2}{\rho^2}.
\end{equation}
Taking \eqref{eq2-20} and \eqref{eq2-21} into \eqref{eq2-18-1},  we obtain
\begin{equation} \label{eq2-22}
	\begin{aligned}
		& \Big( \tilde{\kappa}_1  - \frac{1 + v^2}{v - a} - \frac{\beta v}{u} - \frac{\alpha C}{\rho} \Big) kf - C f^{1-1/(k-1)} \\
		& + \Big( \frac{\beta}{u} + \frac{a - 2 - 4 a^{- 1}}{v - a} - 1- 2 (1 + \delta_1^{- 1}) \frac{\beta^2}{u^2}\\
		& - \frac{ C \alpha + C (3 + 2 \delta_1^{- 1}) \alpha^2 }{\rho^2} -C\Big) \sum \sigma_{k}^{ii} \\
		& + \frac{3 a}{4 (v - a)}  \sigma_{k}^{ii} \tilde{\kappa}_i^2
		- \delta_1 \frac{\sigma_{k}^{ii} (\tilde{\nabla}_{i}v)^2}{(v - a)^2} \leq 0.
	\end{aligned}
\end{equation}
By Lemma \ref{Lemma2-2} and \eqref{eq1-16}, we know that
\[ \tilde{\nabla}_{i}v = \frac{\tilde{\nabla}_{i}u}{u} (v - \tilde{\kappa}_i), \]
and consequently,
\begin{equation} \label{eq2-31}
	\frac{\sigma_{k}^{ii} \tilde{\nabla}_{i}v^2}{(v - a)^2} \leq \frac{2}{(v - a)^2} \sigma_{k}^{ii} \tilde{\kappa}_i^2 + \frac{2}{(v - a)^2} \sum \sigma_{k}^{ii}.
\end{equation}
Taking \eqref{eq2-31} into \eqref{eq2-22} and choosing $\delta_1 = \frac{a^2}{8}$, we have
\begin{equation*}
	\begin{aligned}
		& \Big( \tilde{\kappa}_1  - \frac{1 + v^2}{v - a} - \frac{\beta v}{u} - \frac{\alpha C}{\rho} \Big) kf - C f^{1-1/(k-1)}  +  \frac{a}{2 (v - a)}  \sum \sigma_{k}^{ii} \tilde{\kappa}_i^2  \\
		& + \Big( \frac{\beta}{u} + \frac{\frac{3}{4} a - 2 - \frac{4}{a}}{v - a} - C - \frac{2 (1 + 8 a^{- 2}) \beta^2}{u^2} - C \frac{ \alpha + (3 + 16 a^{- 2}) \alpha^2 }{\rho^2} \Big) \sum \sigma_{k}^{ii}
		\leq 0.
	\end{aligned}
\end{equation*}
By \eqref{GC2-9} and the fact $\sigma_{k}^{ii}\tilde{\kappa}_{i}^{2}>\theta\tilde{\kappa}_{1}^{2}\sum \sigma_{k}^{ii}$,
we obtain an upper bound for $\rho \tilde{\kappa}_1$.

\noindent
\textbf{Case 2.} $\tilde{h}_{kk}< \varepsilon \tilde{h}_{11}$.
Therefore we have $|\tilde{h}_{jj}|\leq C\varepsilon \tilde{h}_{11}$ for $j=k, \ldots, n$. By Lemma \ref{lem-3}, we fix
$\varepsilon$ such that
\begin{equation}\label{case-2'}
	\frac{2}{\tilde{h}_{11}}\sum_{i\geq 2}\sigma_{k-2;i1}(\tilde{\nabla}_i \tilde{h}_{11})^{2}
	-(1+\frac{a^2}{4})\sum_{i\geq 2}\sigma_{k}^{ii}\frac{(\tilde{\nabla}_{i}\tilde{h}_{11})^{2}}{\tilde{h}_{11}^{2}}\geq 0.
\end{equation}
By \eqref{eq2-16} and Cauchy-Schwartz inequality, we have
\begin{equation} \label{eq2-26}
		\sum_{i\geq 2} \sigma_{k}^{ii} \frac{(\tilde{\nabla}_{i}\rho)^2}{\rho^2} \leq
		\frac{2}{\alpha^2} \sum_{i \geq 2} \sigma_{k}^{ii} \Big( \frac{(\tilde{\nabla}_{i}v)^2}{(v - a)^2} + \frac{2 (\tilde{\nabla}_{i}\tilde{h}_{11})^2}{\tilde{h}_{11}^2} + 2 \beta^2 \frac{(\tilde{\nabla}_{i}u)^2}{u^4} \Big),
\end{equation}
\begin{equation} \label{eq2-27}
	\begin{aligned}
		& \frac{\sigma_{k}^{11} (\tilde{\nabla}_{1}\tilde{h}_{11})^2}{\tilde{h}_{11}^2} \leq
		 2 (1 + \delta_2^{- 1}) \beta^2  \sigma_{k}^{11} \frac{(\tilde{\nabla}_{1}u)^2}{u^4}
		  \\
		& + 2 (1 + \delta_2^{- 1}) \alpha^2  \sigma_{k}^{11} \frac{(\tilde{\nabla}_{1}\rho)^2}{\rho^2} +(1 + \delta_2) \sigma_{k}^{11} \frac{(\tilde{\nabla}_{1}v)^2}{(v - a)^2},
	\end{aligned}
\end{equation}
and
\begin{equation} \label{eq2-23}
	\begin{aligned}
		&  \sum_{i\geq 2}\frac{2 b \sigma_{k}^{ii} \tilde{\nabla}_{i}u \tilde{\nabla}_{i}\rho}{u \rho} \\
		\geq &  \sum_{i\geq 2} \frac{2 \sigma_{k}^{ii} \tilde{\nabla}_{i}u}{u} \Big( - \frac{\tilde{\nabla}_{i}\tilde{h}_{11}}{\tilde{h}_{11}} + \frac{\tilde{\nabla}_{i}v}{v - a} \Big) \\
		\geq & - \delta_2^{-1} \sum_{i\geq 2} \sigma_{k}^{ii} - 2 \delta_2 \sum_{i \geq 2} \sigma_{k}^{ii} \Big( \frac{(\tilde{\nabla}_{i}\tilde{h}_{11})^2}{\tilde{h}_{11}^2} + \frac{(\tilde{\nabla}_{i}v)^2}{(v - a)^2} \Big).
	\end{aligned}
\end{equation}
Taking \eqref{eq2-26}--\eqref{eq2-23} into \eqref{eq2-18-1}, we obtain
\begin{equation} \label{eq2-28}
	\begin{aligned}
		& \Big( \tilde{\kappa}_1 - \frac{1 + v^2}{v - a} - \frac{\beta v}{u} - \frac{\alpha C}{\rho} \Big) kf- C f^{1-1/(k-1)} \\
		& + \Big( \frac{\beta}{u} + \frac{a - 2 - 4 a^{- 1}}{v - a}-C-\delta_2^{-1}- \frac{4 \beta^2 }{\alpha u^2} \Big)\sum \sigma_{k}^{ii} \\
		& + \frac{3 a}{4(v - a)} \sum_{i\geq 2} \sigma_{k}^{ii} \tilde{\kappa}_i^2+ (1-2\alpha^{-1}-2\delta_{2})\sum_{i\geq 2}\frac{\sigma_{k}^{ii}(\tilde{\nabla}_{i}v)^2}{(v - a)^2}\\
		&  +  \sigma_{k}^{11}\Big(\frac{3 a}{4(v - a)}  \tilde{\kappa}_1^2 +\frac{2 \alpha  \tilde{\nabla}_{1}u \tilde{\nabla}_{1}\rho}{u \rho}-\delta_2  \frac{(\tilde{\nabla}_{1}v)^2}{(v - a)^2} \\
		& -\alpha\frac{(\tilde{\nabla}_{1}\rho)^2}{\rho^2}- 2 (1 + \delta_2^{- 1}) \beta^2  \frac{(\tilde{\nabla}_{1}u)^2}{u^4}
		- 2 (1 + \delta_2^{- 1}) \alpha^2   \frac{(\tilde{\nabla}_{1}\rho)^2}{\rho^2}	\Big)  \\
		& + \frac{2\sum_{i \geq 2} \sigma_{k-2; i1} (\tilde{\nabla}_i \tilde{h}_{11})^2 }{\tilde{\kappa}_1}- (1+2\delta_{2})\sum_{i\geq 2} \sigma_{k}^{ii}\frac{(\tilde{\nabla}_{i}\tilde{h}_{11})^2}{\tilde{\kappa}_1^2}
		\leq 0.
	\end{aligned}
\end{equation}
Since $\frac{1}{v-a}>1$, Fixing $\delta_{2}<\frac{a}{8}<\frac{1}{2}$, by \eqref{eq2-31},
\[
(1-2\alpha^{-1}-2\delta_{2})\frac{\sigma_{k}^{ii}(\tilde{\nabla}_{i}v)^2}{(v - a)^2}
\geq  -\frac{4\alpha^{-1}}{(v - a)^2}  \sigma_{k}^{ii} \tilde{\kappa}_i^2 - \frac{4\alpha^{-1}}{(v - a)^2} \sum \sigma_{k}^{ii}
\]
and
\[
\frac{1}{4}\frac{a}{v-a}>2\delta_{2}.
\]
By \eqref{case-2'}, provided $\beta>1$ and $\rho<1$, \eqref{eq2-28} becomes
\begin{equation} \label{eq2-32}
	\begin{aligned}
		& \Big( \tilde{\kappa}_1-\beta C - \frac{\alpha C}{\rho} \Big) kf- C f^{1-1/(k-1)}+ \Big( \frac{\beta}{u}- C- \delta_2^{-1}\\
		& - \frac{\beta^2 C }{\alpha} \Big)\sum \sigma_{k}^{ii} + \Big(\frac{3a}{4(v - a)}-\frac{4\alpha^{-1}}{(v - a)^2} \Big) \sum_{i\geq 2} \sigma_{k}^{ii} \tilde{\kappa}_i^2
		 \\
		&
		+  \sigma_{k}^{11}\Big(\frac{a}{2(v - a)}  \tilde{\kappa}_1^2
		  -\frac{\alpha C}{\rho^2} -2\delta_2 - 2 (1 + \delta_2^{- 1})  \frac{\beta^2 }{u^2}   \\
		&  -  (1 + \delta_2^{- 1}) \frac{\alpha^2 C  }{\rho^2}
		\Big)
		\leq 0.
	\end{aligned}
\end{equation}

Then choosing $C<<\beta <<\alpha$, we have
\[
\begin{aligned}
	& \Big( \tilde{\kappa}_1 -\frac{\alpha C}{\rho} \Big) kf+\frac{1}{2}\frac{\beta}{u}\sum \sigma_{k}^{ii}- C f^{1-1/(k-1)} \\
	& +  \sigma_{k}^{11}\Big(\frac{ 3a}{4(v - a)}  \tilde{\kappa}_1^2-\frac{\alpha C}{\rho^2}
	\Big)
	\leq 0.
\end{aligned}
\]
We thus obtain an upper bound for $\rho^{\alpha} \tilde{\kappa}_1$ by \eqref{GC2-9}.

\end{proof}

\end{document}